\documentclass[12pt,a4paper, leqno]{article}
\usepackage{amsfonts}
\usepackage{amssymb, amscd, amsthm, amsmath,latexsym, amstext,mathrsfs}
\usepackage[all]{xy}
\usepackage{graphicx}
\usepackage{makeidx}
\usepackage{fancyhdr}

\usepackage[colorlinks, linkcolor=blue, anchorcolor=green, citecolor=blue]{hyperref}

\newtheorem{thm}{Theorem}[section]
\newtheorem{coro}[thm]{Corollary}
\newtheorem{lemma}[thm]{Lemma}
\newtheorem{prop}[thm]{Proposition}

\theoremstyle{remark}
\newtheorem{remark}[thm]{\textbf{Remark}}

\theoremstyle{definition}

\numberwithin{equation}{thm}
\newcommand{\cd}{\mathrm{cd}}
\newcommand{\Br}{\mathrm{Br}}
\newcommand{\car}{\mathrm{char}}
\newcommand{\Nrd}{\mathrm{Nrd}}

\newcommand{\Cor}{\mathrm{Cor}}
\newcommand{\CH}{\mathrm{CH}}
\newcommand{\tors}{\mathrm{tors}}
\newcommand{\cl}{\mathrm{cl}}
\newcommand{\et}{\text{\'et}}

\newcommand{\dgf}[1]{\langle\, #1\,\rangle}
\newcommand{\pfi}[1]{\langle\!\langle\, #1\,\rangle\!\rangle}
\newcommand{\Pfi}[1]{\langle\!\langle\, #1\,]\!]}

\newcommand{\newpara}{\noindent\refstepcounter{thm}{\bf(\thethm)\;}}

\newcommand{\vp}{\varphi}

\newcommand{\sK}{\mathscr{K}}

\newcommand{\coker}{\mathrm{Coker}}
\newcommand{\im}{\mathrm{Im}}

\def\ker{\mathrm{Ker}}

\newcommand{\N}{\mathbb N}
\newcommand{\Q}{\mathbb{Q}}
\newcommand{\Z}{\mathbb Z}

\newcommand{\al}{\alpha}

\newcommand{\lra}{\longrightarrow}
\newcommand{\ov}[1]{\overline{#1}}
\newcommand{\simto}{\xrightarrow{\sim}}

\newcommand{\id}{\mathrm Id}
\newcommand{\ilim}{\varinjlim}
\newcommand{\Spec}{\mathrm{Spec}}

\newcommand{\bH}{\mathbb H}
\newcommand{\bP}{\mathbb{P}}
\newcommand{\sH}{\mathscr{H}}
\newcommand{\Gal}{\mathrm{Gal}}

\begin{document}
\title{\textbf{Unramified Cohomology of Quadrics in Characteristic Two}}
\author{Yong HU and Peng SUN}
\date{}

\maketitle

\begin{abstract}
Let $F$ be a field of characteristic 2 and let $X$ be a smooth projective quadric of dimension $\ge 1$ over  $F$. We study the unramified cohomology groups with 2-primary torsion coefficients of $X$ in degrees 2 and 3. We determine completely the kernel and the cokernel of the natural map from the cohomology of $F$ to the unramified cohomology of $X$. This extends the results in characteristic different from 2 obtained by Kahn, Rost and Sujatha in the nineteen-nineties.
\end{abstract}




\medskip

\noindent {\bf MSC classification 2020:} 11E04, 14F20, 19E15

\section{Introduction}

Let $F$ be a field. Let $m$ be a positive integer not divisible by the characteristic of $F$. For each $j\ge 1$, the tensor product $\Z/m(j-1):=\mu_m^{\otimes (j-1)}$ of $m$-th roots of unity can be viewed as an \'etale sheave on $F$-schemes. Let $X$ be a proper smooth connected variety over $F$. The unramified cohomology group $H^j_{nr}\bigl(X,\,\Z/m(j-1)\bigr)$ is the group $H^0_{Zar}\bigl(X,\,\sH^j_m(j-1)\bigr)$, where $\sH^j_m(j-1)$ denotes the Zariski sheaf associated to the presheaf $U\mapsto H^j_{\et}(U,\,\Z/m(j-1))$. By taking the direct limit, we can also define
\[
H^j_{nr}\bigl(X,\,(\Q/\Z)'(j-1)\bigr):=\ilim_{\car(F)\nmid m}H^j_{nr}\bigl(X,\,\Z/m(j-1)\bigr)\,.
\]These groups can also be described in terms of residue maps in Galois cohomology, thanks to the Bloch--Ogus theorem on Gersten's conjecture (\cite{BlochOgus74}). As important birational invariants, they have found many important applications, for instance to the rationality problem (see e.g. \cite{Saltman84invent}, \cite{ColliotOjanguren89invent},
\cite{ColliotPirutka16AnnSciENS}), and have been extensively studied in the literature. With the development of the motivic cohomology theory (by Beilinson, Lichtenbaum, Suslin, Voevodsky, et al.),  even more machinery can be applied to compute unramified cohomology nowadays.

When $F$ has characteristic different from 2 and $X$ is a smooth projective quadric, the above unramified cohomology groups are computed up to degree $j\le 4$ by Kahn, Rost and Sujatha in a series of papers (\cite{Kahn95ArchMath}, \cite{KahnRostSujatha98}, \cite{KahnSujatha00JEMS}, \cite{KahnSujatha01Duke}). Some of their results are further developed and used by Izhboldin \cite{Izhboldin01} to solve a number of problems on quadratic forms, including a construction of fields of $u$-invariant 9 in characteristic $\neq 2$.

It has been noticed for decades that unramified cohomology theory can be formulated in a more general setting (see \cite{CT95}, \cite{ColliotHooblerKahn97}, \cite{Kahn04LMN1835}). In particular, when $F$ has positive characteristic $p$, the aforementioned groups have $p$-primary torsion variants. Indeed, the unramified cohomology functors $H^j_{nr}\bigl(\cdot\,,\,\Z/p^r(j-1)\bigr)$ for all $r\ge 1$ and their limit $H^j_{nr}\bigl(\cdot\,, \Q_p/\Z_p(j-1)\bigr)$ can be defined by using the Hodge--Witt cohomology (see \eqref{3p1temp} and \eqref{3p10temp} for a brief review). In contrast to the prime-to-$p$ case, there has been much less work on these $p$-primary unramified cohomology groups.

\medskip

In this paper, we are interested in the case of a smooth projective quadric $X$ over a field $F$ of characteristic 2. We investigate the unramified cohomology groups via the natural maps
\[
\eta^j_r\,:\; H^j\bigl(F,\,\Z/2^r(j-1)\bigr)\lra H^j_{nr}\bigl(X,\,\Z/2^r(j-1)\bigr)\,,\quad r\ge 1\,
\]and
\[
\eta^j_\infty\,:\; H^j\bigl(F,\,\Q_2/\Z_2(j-1)\bigr)\lra H^j_{nr}\bigl(X,\,\Q_2/\Z_2(j-1)\bigr)\,.
\]
For each $j\ge 1$, it is not difficult to see that the maps $\eta^j_r$ for different $r\ge 1$ have essentially the same behavior (Lemma\;\ref{4p3temp}). So we may focus on the two maps $\eta^j:=\eta_1^j$ and $\eta^j_{\infty}$. They are both isomorphisms if $j=1$ (Prop.\;\ref{3p4temp}) or $X$ is isotropic (Prop.\;\ref{4p1temp} (1)). In our main results, we determine completely the kernel and the cokernel of $\eta^j$ and $\eta^j_{\infty}$ for $j=2,\,3$.

\medskip

For any $c\in F$, we denote by $(c]$ its canonical image in the quotient $F/\wp(F)$ (where $\wp$ is the map $x\mapsto x^2-x$). It is the $e_1$-invariant (or Arf invariant) of the quadratic Pfister form $\langle\!\langle c]]: (x,\,y)\mapsto x^2+xy+cy^2$. A similar and perhaps more familiar notation is $(a)\in F^*/F^{*2}$, which we use to denote the canonical image of an element $a\in F^*$.

%

\medskip

The following theorem extends Kahn's results in \cite{Kahn95ArchMath} to characteristic 2.

\begin{thm}
  Let $F$ be a field of characteristic $2$ and let $X$ be the smooth projective quadric defined by a nondegenerate quadratic form $\vp$ with $\dim \vp\ge 3$. Assume that $\vp$ is anisotropic.

  \begin{enumerate}
    \item (See $\ref{5p1temp}, \ref{5p3temp}, \ref{5p6temp}$ and $\ref{5p7temp}$) Suppose $\dim\vp=3$, so that $X$ is the conic associated to a quaternion division algebra $D$, or $\dim\vp=4$ and $e_1(\vp)=0\in H^1(F,\,\Z/2)$, that is, $\vp$ is similar to the reduced norm of a quaternion division algebra $D$.

    Then
    \[
    \begin{split}
      \coker(\eta^2)&\cong \ker(\eta^2)=\{0,\,(D)\}\cong \Z/2\,,\\
      \coker(\eta^3)&\cong \ker(\eta^3)=\{(a)\cup (D)\,|\,a\in F^*\}\,,
    \end{split}
    \]where $(D)$ denotes the Brauer class of $D$.
    \item (See  $\ref{5p8temp}$ and $\ref{6p6temp}$) If $\dim\vp=4$ and $e_1(\vp)\neq 0\in H^1(F,\,\Z/2)$, then
    \[
    \begin{split}
      \coker(\eta^2)&=\ker(\eta^2)=0\,,\; \coker(\eta^3)\cong \ker(\eta^3)\,,\\
     \ker(\eta^3)&=\{0\}\cup \{(a)\cup (b)\cup (c]\,|\,\vp \text{ is similar to a subform of } \langle\!\langle a,\,b\,;\,c]]\}\,.
    \end{split}
    \]
    \item (See $\ref{6p6temp}, \ref{6p10temp}, \ref{6p11temp}$ and $\ref{8p4temp}$) Suppose $\dim\vp>4$. Then $\coker(\eta^2)=\ker(\eta^2)=0$ and

    \begin{enumerate}

      \item if $\vp$ is an Albert form (i.e. a $6$-dimensional form with trivial Arf invariant), then $\coker(\eta^3)\cong \Z/2$ and $\ker(\eta^3)=0$;

      \item if $\vp$ is a neighbor of a $3$-Pfister form $\langle\!\langle a,\,b\,;\,c]]$, then
      \[
      \coker(\eta^3)\cong\ker(\eta^3)=\{0,\,(a)\cup (b)\cup (c]\}\cong \Z/2\,.
      \]
      \item in all the other cases (e.g. $\dim\vp>8$), $\coker(\eta^3)=\ker(\eta^3)=0$.
    \end{enumerate}
  \end{enumerate}
\end{thm}

The counterpart in characteristic $\neq 2$ of the next theorem appeared in \cite[Thms.\,4 and 5]{KahnRostSujatha98}.

\begin{thm}[See $\ref{4p3temp},\,\ref{6p7temp}$ and $\ref{8p5temp}$]
  Let $F$ be a field of characteristic $2$ and let $X$ be the smooth projective quadric defined by a nondegenerate quadratic form $\vp$ with $\dim \vp\ge 3$.

  Then $\ker(\eta^j_r)=\ker(\eta^j_{\infty})$ for all $r\ge 1,\,j\ge 1$, and $\coker(\eta^2_{\infty})=0$.

  If $\vp$ is an anisotropic Albert form, then  $\coker(\eta^3_{\infty})\cong \Z/2$. Otherwise $\coker(\eta^3_{\infty})=0$.
\end{thm}

As in \cite{Kahn95ArchMath} and \cite{KahnRostSujatha98}, main tools in our proofs include the Bloch--Ogus and the Hochschild--Serre spectral sequences. A key difference between the $p$-primary torsion cohomology and the prime-to-$p$ case is the lack of homotopy invariance. This results in the phenomenon that our spectral sequences look different from their analogues in characteristic different from 2. Due to vanishing theorems for the local cohomology and the fact that the field $F$ has 2-cohomological dimension at most 1, these spectral sequences still have many vanishing terms.

Cycle class maps with finite or divisible coefficients are also studied and used in the paper. In this respect we need information about the structure of Chow groups in low codimension. This information can be found in Karpenko's work \cite{Karpenko90AlgGeoInv} in characteristic $\neq 2$, and recently the paper \cite{HLS21} has provided the corresponding results in characteristic 2.

We mention a situation where the distinction of characteristic affects the study of a cycle class map, and hence also makes a remarkable difference in our proofs. Suppose the quadric $X$ is defined by a neighbor of a 3-Pfister form and consider the cycle class map $\mathrm{cl}^2_X$ on the codimension two Chow group. In characteristic different from 2,  the image of the torsion element under $\mathrm{cl}^2_X$ can be described as a cup product element (\cite[Prop.\;5.4.6]{Shyevski90}). The lack of such a description forces us to proceed differently for results concerning $\coker(\eta^3)$ and $\coker(\eta^3_{\infty})$ (see the proofs of Theorem\;\ref{6p11temp} and Corollary\;\ref{8p5temp}).

\medskip

In the above two theorems the case of Albert quadrics is more subtle than the others. In that case we have to utilize more techniques from the algebraic theory of quadratic forms, especially residue maps on Witt groups of discrete valuation fields of characteristic 2 (\cite{Arason18}).

\

\noindent {\bf Notation and conventions.} For any field $k$, denote by $\ov{k}$ a separable closure of $k$.

For an algebraic variety $Y$ over $k$,  we write $Y_L=Y\times_kL$ for any field extension $L/k$, and  $\ov{Y}=Y\times_k\ov{k}$.  We say $Y$ is $k$-\emph{rational} if it is integral and birational to the projective space $\bP^{\dim Y}_k$ over $k$. We say $Y$ is \emph{geometrically rational} if $Y_L$ is $L$-rational for the algebraic closure $L$ of $k$.

Milnor $K$-groups of a field $k$ are denoted by $K^M_i(k),\,i\in\N$.

For an abelian group $M$, we denote by $M_{\tors}$ the subgroup of torsion elements in $M$. For any positive integer $n$, we define $M[n]$ and $M/n$ via the exact sequence
\[
0\lra M[n]\lra M\overset{\times n}{\lra} M\lra M/n\lra 0\,.
\]

For any scheme $X$, let $\Br(X)=H^2_{\et}(X,\,\mathbb{G}_m)$ denote its cohomological Brauer group.

In the rest of the paper,   $F$ denotes a field of characteristic 2.

\section{Quadrics and their Chow groups}

\newpara\label{2p1temp} We recall some basic definitions and facts about quadratic forms in characteristic 2. For general reference we refer to \cite{EKM08}.

We  work over a field $F$, which has characteristic 2 according to our convention.

Let $a\in F^*$. We denote by $\langle a\rangle$ the 1-dimensional quadratic form $x\mapsto ax^2$,  and let  $[1,\,a]$ or $\langle\!\langle a]]$ denote the binary quadratic form $(x,\,y)\mapsto x^2+xy+ay^2$.
A (quadratic) \emph{$1$-Pfister form} is a binary quadratic form isomorphic to $\langle\!\langle a]]=[1,\,a]$ for some $a\in F^*$. Let $\langle 1,\,a\rangle_{\mathrm{bil}}$ denote the binary bilinear form $((x_1,\,x_2), (y_1, \,y_2))\mapsto x_1y_1+ax_2y_2$. For $n\ge 2$, a quadratic form  is called an \emph{$n$-Pfister form} if it is isomorphic to
\[
\langle\!\langle a_1,\cdots, a_{n-1}\,;\,a_n]]:=\langle 1,\,a_1\rangle_{\mathrm{bil}}\otimes \cdots\otimes\langle 1,\,a_{n-1}\rangle_{\mathrm{bil}}\otimes \langle\!\langle a_n]]\,
\]for some $a_1,\cdots, a_{n-1}\in F^*$ and $a_n\in F$.

For two quadratic forms $\vp$ and $\psi$ over $F$, we say $\psi$ is a \emph{subform} of $\vp$ if $\psi\cong\vp|_{W}$ for some subspace $W$ in the vector space $V$ of $\vp$.
For $n\ge 2$, an \emph{$n$-Pfister neighbor} is a form of dimension $>2^{n-1}$ that is similar to a subform of an $n$-Pfister form.

We write $I_q(F)$ or $I_q^1(F)$ for the Witt group of even-dimensional nondegenerate quadratic forms over $F$. For $n\ge 2$, let $I^n_q(F)$ denote the subgroup of $I_q(F)$ generated by the $n$-Pfister forms. For a quadratic form $\vp$ over $F$, we will write $\vp\in I^n_q(F)$ if $\vp$ is nondegenerate, of even dimension, and its Witt class lies in $I^n_q(F)$.

We also have the Witt ring $W(F)$ of nondegenerate symmetric bilinear forms over $F$, in which the classes of even-dimensional forms form an ideal $I(F)$, called the \emph{fundamental ideal}. For each $n\ge 1$, let $I^n(F)$ be the $n$-th power of the ideal $I(F)$ and put $I^0(F)=W(F)$. The group $I_q(F)$ has a $W(F)$-module structure, and we have $I^n_{q}(F)=I^{n-1}(F)\cdot I_q(F)$ for all $n\ge 1$.

The Galois cohomology group $H^1(F,\,\Z/2)$ can be identified with $F/\wp(F)$ by Artin--Schreier theory, where $\wp$ denotes the map $x\mapsto x^2-x$. For any $b\in F$, we write $(b\,]$ for its canonical image in $F/\wp(F)=H^1(F,\,\Z/2)$. The map
\[
e_1\,:\; I^1_q(F)\lra H^1(F,\,\Z/2)\,;\quad \langle\!\langle a]]\longmapsto (a\,]
\]is a well defined homomorphism,  often called the \emph{discriminant} or \emph{Arf invariant}. It is well known that $e_1$ is surjective with $\ker(e_1)=I^2_q(F)$. A 6-dimensional form in $I^2_q(F)$ (i.e. a 6-dimensional nondegenerate form with trivial Arf invariant) is called an \emph{Albert form}.

For $n\ge 2$, by using the Kato--Milne group $H^n\bigl(F,\,\Z/2(n-1)\bigr)$ (cf. \cite{KatoII80}, \cite{Milne76AnnSciENS}), a generalization of which will be discussed in \eqref{3p1temp}, one can also define a functorial homomorphism (see \cite{Sah72JA} for $n=2$ and \cite{Kato82Invent} for general $n$)
\[
e_n\,:\;I^n_q(F)\lra H^n\bigl(F,\,\Z/2(n-1)\bigr)
\]which is surjective with $\ker(e_n)=I^{n+1}_q(F)$ such that
\[
e_n\bigl(\langle\!\langle a_1,\cdots, a_{n-1}\,;\,a_n]]\bigr)=(a_1)\cup \cdots\cup (a_{n-1})\cup (a_n\,]\,,
\]where for any $a\in F^*$, $(a)$ denotes its canonical image in $F^*/F^{*2}$. (The maps $e_2$ and $e_3$ are more classical, called the \emph{Clifford invariant} and the \emph{Arason invariant} respectively.)

Note that for all $n\ge 1$,
\begin{equation}\label{2p1p1temp}
  e_n(\vp)=e_n(c.\vp)\quad\text{for all } \vp\in I^n_q(F)\quad\text{and all}\quad c\in F^*
\end{equation}
because $\vp-c\vp=\pfi{c}\vp\in I^{n+1}_q(F)=\ker(e_n)$.

The group $H^2\bigl(F,\,\Z/2(1)\bigr)$ may be identified with the 2-torsion subgroup of the Brauer group $\Br(F)$ of $F$. For any $a\in F^*$ and $b\in F$, let $(a,\,b\,]$ be the quaternion algebra generated by two elements $i,\,j$ subject to the relations
\[
i^2=a\,,\,j^2+j=b\,,\quad ij=ji+i\,.
\]Its reduced norm is the 2-Pfister form $\langle\!\langle a;\,b]]$, and its Brauer class is $(a)\cup (b\,]$, the Clifford invariant of $\langle\!\langle a;\,b]]$. The plane conic defined by the ternary form $\langle a\rangle\bot \langle\!\langle b]]$ is called the conic associated to the quaternion algebra $(a,\,b\,]$.

For any $a,\,c\in F^*$ and $b,\,d\in F$, the form $[1,\,b+d]\bot a.[1,\,b]\bot c.[1,\,d]$ is an Albert form. Its Clifford invariant is the Brauer class of the biquaternion algebra
$(a,\,b]\otimes_F(c,\,d]$.

\

\newpara\label{2p2temp} Now we recall some known facts about Chow groups of projective quadrics (which are valid in arbitrary characteristic). More details can be found in \cite[\S\,2]{Karpenko90AlgGeoInv}, \cite[\S\,68]{EKM08} and \cite[\S\,5]{HLS21}.

Let $\vp$ be a quadratic form of dimension $\ge 3$ over $F$. Let $X$ be the projective quadric defined by $\vp$. Unless otherwise stated, we always assume $\vp$ is nondegenerate, which means $X$ is smooth as an algebraic variety over $F$.

For each $i\in\N$, let $\CH^i(X)$ denote the Chow group of codimension $i$ cycles of $X$. Let $h\in \CH^1(X)$ be the class of a hyperplane section. Using the intersection pairing as multiplication in the Chow ring (\cite[\S\,57]{EKM08}), we get elements $h^i\in \CH^i(X)$ for each $i$. Set $d=\dim X$. For every integer $j\in [0,\,d/2]$, let $\ell_j\in\CH^{d-j}(\ov{X})$ be the class of a $j$-dimensional linear subspace contained in $\ov{X}$. Then, for each $0\le i\le d$, we have
		\[
		\CH^i(\ov{X})=
		\begin{cases}
			\Z.h^i\quad & \text{ if }\;\; 0\le i<\frac{d}{2}\,,\\
			\Z.\ell_{d-i} & \text{ if }\;\; \frac{d}{2}< i\le d\,,\\
			\Z.h^i\oplus \Z.\ell_i & \text{ if }\;\; i=\frac{d}{2}\,.
		\end{cases}
		\]
If $\dim X=2m$ is even, there are exactly two different classes of $m$-dimensional linear subspaces $\ell_m,\,\ell_m'$ in $\CH^m(\ov{X})$ and the sum of these two classes is equal to $h^m$. Moreover, $\CH^m(\ov{X})$ is a trivial Galois module if and only if $e_1(\vp)=0$ (\cite[Lemma\;8.2]{Kahn99MotivicCellular}). When $e_1(\vp)\neq 0$, the Galois action permutes the two classes $\ell_m$ and $\ell'_m$.

\section{Unramified cohomology in positive characteristic}\label{sec3}

Throughout this section, we fix a positive integer $r$ and a field $k$ of characteristic $p>0$.

\medskip

\newpara\label{3p1temp}  For each $i\in\N$, let $\nu_r(i)=W_r\Omega^i_{\log}$ be the $i$-th logarithmic Hodge--Witt sheaf on the big \'etale site of $k$ (\cite{Illusie79}, \cite{Shiho07}). Define $\Z/p^r(i):=\nu_r(i)[-i]$, as an object in the derived category of \'etale sheaves. This object can also be viewed as an \'etale motivic complex (\cite{GeisserLevine00InventMath}).  Without giving details of the constructions, let us mention that our whole paper relies on the fact that $\Z/p^r(i)$ is the correct analogue of the more frequently used \'etale sheaf $\Z/m(i)$ for $m$ prime to $p$, from both the \'etale cohomological and the $K$-theoretic points of view. In particular, the Bloch--Kato--Gabber theorem (\cite{BlochKato86}) plays a crucial role in the subsequent discussions.

For every integer $b$, we have  the cohomology functor on $k$-schemes
\[
H^b\bigl(\cdot\,,\,\Z/p^r(i)\bigr):=H^b_{\et}\bigl(\cdot\,,\,\Z/p^r(i)\bigr)=H^{b-i}_{\et}\bigl(\cdot\,,\,\nu_r(i)\bigr)\,.
\]For shorthand, we sometimes write $H^b_{p^r}(\cdot,\,i)$ instead of the precise notation $H^b\bigl(\cdot\,,\,\Z/p^r(i)\bigr)$.
The Zariski sheaf associated to the presheaf $U\mapsto H^b\bigl(U,\,\Z/p^r(i)\bigr)$ is denoted by $\sH^b_{p^r}(i)$.

For a smooth connected $k$-variety $X$, we define the \emph{unramified cohomology group}
\[
H^b_{nr}\bigl(X\,,\,\Z/p^r(i)\bigr):=H^0_{Zar}\bigl(X\,,\,\sH^b_{p^r}(i)\bigr)\,.
\]This group can also be described by using the Cousin complex of $X$. It is naturally a subgroup of $H^b\bigl(k(X),\,\Z/p^r(i)\bigr)$.

\

In the theorem below, we collect some well known results that are needed in this paper. It is worth noticing that \eqref{3p2p2temp} is a special phenomenon in characteristic $p$.
(It fails dramatically in characteristic $\neq p$.)

\begin{thm}\label{3p2temp}
	Let $X$ be a smooth connected $k$-variety.
	
	\begin{enumerate}
		\item We have the \emph{Bloch--Ogus spectral sequence}
		\begin{equation}\label{3p2p1temp}
			E_2^{a,\,b}=H^a_{Zar}\big(X\,,\,\sH^b_{p^r}(i)\big)\Longrightarrow E^{a+b}=H^{a+b}\big(X\,,\,\Z/p^r(i)\big)
		\end{equation}
		with
		\begin{equation}\label{3p2p2temp}
			E_2^{a,\,b}=0\quad \text{ if }\; b\notin\{i,\,i+1\}\,, \text{ or  if }\; a>b=i
		\end{equation}
		and
		\begin{equation}\label{3p2p3temp}
			E_2^{i,\,i}=H^i_{Zar}\big(X\,,\,\sH^i_{p^r}(i)\big)\cong \CH^i(X)/p^r\;.
		\end{equation}
		\item There are natural isomorphisms
		\begin{equation}\label{3p2p4temp}
			\begin{split}
				H^i\big(X\,,\,\Z/p^r(i)\big)&\simto H^0_{Zar}\big(X\,,\,\sH^i_{p^r}(i)\big)=H^i_{nr}\big(X\,,\,\Z/p^r(i)\big)\,,\\
				H^{2i+j}\big(X\,,\,\Z/p^r(i)\big)&\simto H^{j+i-1}_{Zar}\big(X\,,\,\sH^{i+1}_{p^r}(i)\big)\quad \text{ for }\; j\ge 1\,,
			\end{split}
		\end{equation}
		and
		\begin{equation}\label{3p2p5temp}
			H^2_{nr}(X,\,\Z/p^r(1))=H^0_{Zar}\big(X,\,\sH^2_{p^r}(1)\big)\cong \Br(X)[p^r]\;.
		\end{equation}
		\item For smooth proper connected $k$-varieties, the group $H^b_{nr}\big(X\,,\,\Z/p^r(i)\big)$ is a $k$-birational invariant.
		\item Let $\pi: X\to Y$ be a proper morphism between smooth connected $k$-varieties whose generic fiber is $k(Y)$-rational.
		
		Then the natural map $\pi^*: H^b_{nr}\big(Y,\,\Z/p^r(i)\big)\to H^b_{nr}\big(X\,,\,\Z/p^r(i)\big)$ is an isomorphism.
	\end{enumerate}
\end{thm}
\begin{proof}
	 (1) These are standard consequences of Gersten's conjecture for smooth varieties (\cite[Cor.\;5.1.11 and \S\,7.4\,(3)]{ColliotHooblerKahn97}, \cite[Thm.\;4.1]{Shiho07}) and the Bloch--Kato--Gabber theorem (\cite[\S\,2]{BlochKato86}). The nontrivial part of \eqref{3p2p2temp} is proved in \cite[(3.5.3)]{Gros85} and \cite[Cor.\;3.4]{Shiho07}.

	(2) The isomorphisms in \eqref{3p2p4temp} follow easily from the vanishing results in \eqref{3p2p2temp}. The isomorphism \eqref{3p2p5temp} is a consequence of purity for Brauer groups (\cite[Thm.\;1.2]{Cesnavicius19Duke}).
	
	(3) See \cite[Thm.\;8.5.1]{ColliotHooblerKahn97}.
	
	(4) See \cite[Thm.\;8.6.1]{ColliotHooblerKahn97}.
\end{proof}

\newpara\label{3p3temp} Let $X$ be a smooth connected $k$-variety. For every  $j\ge 1$, we have a
natural restriction map
\begin{equation}\label{3p3p1temp}
	\eta^{j}_r\;:\; H^{j}\bigl(k,\,\Z/p^r(j-1)\bigr)\lra H^j_{nr}\bigl(X,\,\Z/p^r(j-1)\bigr)\,.
\end{equation}
Let $i\in\N$ be such that $i\le j$. We have the cup product map
\begin{equation}\label{3p3p2temp}
	H^{j-i}_{nr}\bigl(X\,,\,\Z/p^r(j-i)\bigr)\otimes H^i_{Zar}\bigl(X,\,\sH^i_{p^r}(i)\bigr)\overset{\cup}{\lra} H^i_{Zar}\bigl(X\,,\,\sH^j_{p^r}(j)\bigr)
\end{equation}and a natural map
\begin{equation}\label{3p3p3temp}
	H^{j-i}\bigl(k\,,\,\Z/p^r(j-i)\bigr)\otimes \CH^i(X)/p^r\lra H^{j-i}_{nr}\bigl(X\,,\,\Z/p^r(j-i)\bigr)\otimes H^i_{Zar}\bigl(X,\,\sH^i_{p^r}(i)\bigr)
\end{equation}
induced from the identification $\CH^i(X)/p^r=H^i_{Zar}\big(X,\,\sH^i_{p^r}(i)\big)$ in \eqref{3p2p3temp} and the natural map
\[
H^{j-i}\bigl(k,\,\Z/p^r(j-i)\bigr)\lra H^0_{Zar}\bigl(X,\,\sH^{j-i}_{p^r}(j-i)\bigr)=H^{j-i}_{nr}\bigl(X\,,\,\Z/p^r(j-i)\bigr)\;.
\]
Composing \eqref{3p3p2temp} with \eqref{3p3p3temp} yields a natural map
\begin{equation}\label{3p3p4temp}
	\mu^{i,\,j}_r\;:\; H^{j-i}\bigl(k\,,\,\Z/p^r(j-i)\bigr)\otimes \CH^i(X)/p^r\lra H^i_{Zar}\big(X\,,\,\sH^j_{p^r}(j)\big)\;.
\end{equation}
For $i=0$ this coincides with the natural map
\[
\tilde{\eta}^j_r\,:\;H^j\bigl(k\,,\,\Z/p^r(j)\bigr)\to H^j_{nr}\bigl(X,\,\Z/p^r(j)\bigr)=H^j\bigl(X,\,\Z/p^r(j)\bigr)\,.
\]

\begin{prop}\label{3p4temp}
	With notation as in $\eqref{3p3temp}$, suppose that $X$ is proper and geometrically rational.

	Then the natural map $\eta^1_r\;:\; H^1(k,\,\Z/p^r)\to H^1_{nr}\big(X,\,\Z/p^r\big)$ is an isomorphism.
\end{prop}
\begin{proof}
	By \eqref{3p2p4temp}, we have $H^1_{nr}(X,\,\Z/p^r)=H^1(X,\,\Z/p^r)$. The result is a well known consequence of the homotopy exact sequence for \'etale fundamental groups (\cite[(IX.6.1)]{SGA1}).
\end{proof}

\begin{prop}\label{3p5temp}
	Let $X$ be a smooth, proper, $k$-rational variety.
	
	Then for every $j\ge 1$, the map
\[
\eta^j_r\;:\; H^j\bigl(k,\,\Z/p^r(j-1)\bigr)\lra H^j_{nr}\big(X,\,\Z/p^r(j-1)\big)
\] is an isomorphism.
\end{prop}
\begin{proof}
	This is a special case of Theorem\;\ref{3p2temp} (4).
\end{proof}

\newpara\label{3p6temp} Let $X$ be a smooth connected $k$-variety. The  Bloch--Ogus spectral sequence \eqref{3p2p1temp} is concentrated in two horizontal lines by \eqref{3p2p2temp} (which relies particularly heavily on the characteristic $p$ assumption). So we can obtain natural homomorphisms
\[
e^a(i)\,:\; E_2^{a,\,i}=H^a_{Zar}\bigl(X\,,\,\sH^i_{p^r}(i)\bigr)\lra E^{i+a}=H^{i+a}\bigl(X,\,\Z/p^r(i)\bigr)\;,\quad 1\le a\le i\,
\]which fit into a long exact sequence
\begin{equation}\label{3p6p1temp}
  \begin{split}
    0\lra &\;E_2^{1,\,i}\xrightarrow[]{e^1(i)} E^{i+1}\lra E_2^{0,\,i+1}=H^0_{Zar}\big(X,\,\sH^{i+1}_{p^r}(i)\big)\\
    \overset{d}{\lra}&\;E_2^{2,\,i}\xrightarrow[]{e^2(i)} E^{i+2}\lra E_2^{1,\,i+1}=H^1_{Zar}\big(X\,,\,\sH^{i+1}_{p^r}(i)\big)\to\cdots\\
      \cdots&\cdots\cdots\cdots\cdots\cdots\\
      \overset{d}{\lra}&\;E_2^{a,\,i}\xrightarrow[]{e^a(i)} E^{i+a}\lra E_2^{a-1,\,i+1}=H^{a-1}_{Zar}\big(X\,,\,\sH^{i+1}_{p^r}(i)\big)\to\cdots\\
      \cdots&\cdots\cdots\cdots\cdots\cdots\\
    \overset{d}{\lra}&\;E_2^{i,\,i}\xrightarrow[]{e^i(i)} E^{2i}\lra E_2^{i-1,\,i+1}=H^{i-1}_{Zar}\big(X\,,\,\sH^{i+1}_{p^r}(i)\big)\lra 0\;.
  \end{split}
\end{equation}
In particular, we have a natural map
\begin{equation}\label{3p6p2temp}
	\cl^i_X=\cl^i:=e^i(i)\;:\; E_2^{i,\,i}=\CH^i(X)/p^r\lra E^{2i}=H^{2i}(X\,,\,\Z/p^r(i))
\end{equation}
called the \emph{cycle class map}.

Let $j\in\N$ be another integer with $j\ge i$. By composing $e^i(j)$ with the map $\mu^{i,\,j}_r$ in \eqref{3p3p4temp} we get a natural map
\begin{equation}\label{3p6p3temp}
	\nu^{i,\,j}_r\;:\; H^{j-i}\bigl(k,\,\Z/p^r(j-i)\bigr)\otimes \CH^i(X)/p^r\lra H^{i+j}\bigl(X\,,\,\Z/p^r(j)\bigr)\,.
\end{equation}Compatibility of the Bloch--Ogus spectral sequence with cup products (cf.\,\cite[p.868]{KahnRostSujatha98}) gives the commutative diagram
{
	\begin{equation}\label{3p6p4temp}
		\xymatrix{
			&  H^{j-i}\bigl(k,\,\Z/p^r(j-i)\bigr)\otimes\CH^i(X)/p^r \ar[ddl]_{\tilde{\eta}_r^{j-i}\otimes \id} \ar[dddddl]^{\mu_r^{i,\,j}} \ar[dddddr]_{\nu_r^{i,\,j}} \ar[ddr]^{\tilde{\eta}_r^{j-i}\otimes \cl^i_X} &\\
			&&&\\
			S \ar[ddd]_{\cup}\ar@{.>}[rr]^{\id\otimes e^i(i)} && T  \ar[ddd]_{\cup}\\
			&&&\\
			&&&\\
			H^i_{Zar}\bigl(X\,,\,\sH^j_{p^r}(j)\bigr) \ar[rr]^{e^i(j)} && H^{i+j}\bigl(X\,,\,\Z/p^r(j)\bigr)
		}
	\end{equation}
}
where
\[
\begin{split}
  S&=H^{j-i}_{nr}\bigl(X\,,\,\Z/p^r(j-i)\big)\otimes H^i_{Zar}\bigl(X\,,\,\sH^i_{p^r}(i)\bigr)\,,\\
T&=H^{j-i}_{nr}\bigl(X\,,\,\Z/p^r(j-i)\bigr)\otimes H^{2i}\bigl(X,\,\Z/p^r(i)\bigr) \,.
\end{split}
\]

\begin{prop}\label{3p7temp}
	Let $X$ be a smooth connected $k$-variety. Then there is an exact sequence
	\[
	0\lra \CH^1(X)/p^r\overset{\cl^1_X}{\lra}H^2_{\et}(X,\,1)\lra \Br(X)[p^r]\lra 0\,.
	\]
	In particular, the cycle class map $\cl^1_X$ is injective, and if $X$ is proper and $k$-rational, we have $\coker(\cl^1_X)\cong \Br(k)[p^r]$.
\end{prop}
\begin{proof}
	Note that $H^2_{Zar}\big(X,\,\sH^1(1)\big)=0$ by \eqref{3p2p2temp}. Thus, taking $i=1$ in \eqref{3p6p1temp} and noticing the isomorphism \eqref{3p2p5temp} yields the result.
\end{proof}

\begin{prop}\label{3p8temp}
	Let $X$ be a smooth proper $k$-rational variety. Then there is an exact sequence
	\begin{equation}\label{3p8p1temp}
		\begin{split}
			0&\lra H^1_{Zar}\bigl(X\,,\,\sH^2_{p^r}(2)\bigr)\lra H^3\bigl(X,\,\Z/p^r(2)\bigr)\lra    H^3\bigl(k,\,\Z/p^r(2)\bigr)\\
			&\lra \CH^2(X)/p^r\overset{\cl^2_X}{\lra}H^4\bigl(X,\,\Z/p^r(2)\bigr)\lra H^1_{Zar}\big(X,\,\sH^3_{p^r}(2)\big)\lra 0\,.
		\end{split}
	\end{equation}
	In particular, if $H^3\bigl(k,\,\Z/p^r(2)\bigr)=0$ (e.g. if $k$ is separably closed), then the cycle class map
\[
\cl^2_X: \CH^2(X)/p^r\lra H^{4}\big(X,\,\Z/p^r(2)\big)
\] is injective.
\end{prop}
\begin{proof}
	Taking $i=2$ in \eqref{3p6p1temp} we obtain an exact sequence
	\begin{equation}\label{3p8p2temp}
		\begin{split}
			0&\lra H^1_{Zar}\bigl(X\,,\,\sH^2_{p^r}(2)\bigr)\lra H^3\bigl(X,\,\Z/p^r(2)\bigr)\lra       H^3_{nr}\bigl(X,\,\Z/p^r(2)\bigr)\\
			&\lra \CH^2(X)/p^r\overset{\cl^2}{\lra}H^4\bigl(X,\,\Z/p^r(2)\bigr)\lra H^1_{Zar}\big(X,\,\sH^3_{p^r}(2)\big)\lra 0\,.
		\end{split}
	\end{equation}
	The desired exact sequence follows from \eqref{3p8p2temp}, because for the (smooth proper) $k$-rational variety $X$ we have  $H^3_{nr}\bigl(X,\,\Z/p^r(2)\bigr)\cong H^3\bigl(k,\,\Z/p^r(2)\bigr)$ (Prop.\;\ref{3p5temp}).
\end{proof}

\newpara\label{3p9temp} Let $X$ be a smooth connected $k$-variety. For each $i\in\N$ we have the Hochschild--Serre spectral sequence
\begin{equation}\label{3p9p1temp}
	H^a\bigl(k\,,\,H^b(\ov{X}\,,\,\Z/p^r(i))\bigr)\Longrightarrow H^{a+b}\bigl(X\,,\,\Z/p^r(i)\bigr)\;.
\end{equation}
Since $\cd_p(k)\le 1$, we have	$H^a\bigl(k\,,\,H^b(\ov{X}\,,\,\Z/p^r(i))\bigr)=0$ for all $a>1$. (Here again, there is a  significant difference with the prime-to-$p$ cohomology theory.) Thus, the spectral sequence \eqref{3p9p1temp} yields an isomorphism
\begin{equation}\label{3p9p2temp}
	H^0\bigl(k\,,\,H^i(\ov{X}\,,\,\Z/p^r(i))\bigr) \cong H^i\bigl(X\,,\,\Z/p^r(i)\bigr)
\end{equation}
and an exact sequence
\begin{equation}\label{3p9p3temp}
\begin{split}
  	0\lra H^1\bigl(k\,,\,H^j(\ov{X},\,\Z/p^r(i))\bigr)&\lra H^{j+1}\bigl(X,\,\Z/p^r(i)\bigr)\\
  &\lra H^0\bigl(k\,,\,H^{j+1}(\ov{X}\,,\,\Z/p^r(i))\bigr)\lra 0\;
\end{split}
\end{equation}for each $j\ge i$.

Now assume $X$ is proper and geometrically rational (e.g. $X$ is a smooth projective quadric of dimension $\ge 1$). Then by \eqref{3p2p4temp} and Thm.\;\ref{3p2temp} (4), we have
\begin{equation}\label{3p9p4temp}
	H^i\bigl(\ov{X}\,,\,\Z/p^r(i)\bigr)\cong H^i_{nr}\bigl(\ov{X}\,,\,\Z/p^r(i)\bigr)=H^i\bigl(\ov{k}\,,\,\Z/p^r(i)\bigr)=K_i^M(\ov{k})/p^r\;.
\end{equation}
Here the identification to the Milnor $K$-theory in the last equality of \eqref{3p9p4temp} is given by the Bloch--Kato--Gabber theorem (\cite[Cor.\;2.8]{BlochKato86}).
It is proved in \cite[Cor.\;6.5]{Izhboldin91} that
\begin{equation}\label{3p9p5temp}
\begin{split}
	H^i\bigl(k,\,\Z/p^r(i)\bigr)&=K_i^M(k)/p^r=H^0\bigl(k\,,\,K^M_i(\ov{k})/p^r\bigr)\;,\\
 H^{i+1}\bigl(k,\,\Z/p^r(i)\bigr)&=H^1\bigl(k\,,\,K^M_i(\ov{k})/p^r\bigr)\;.
\end{split}
\end{equation}
Hence, combining \eqref{3p9p4temp}, \eqref{3p9p5temp} and the case $j=i$ of \eqref{3p9p3temp}, we obtain an exact sequence
\begin{equation}\label{3p9p6temp}
	0\lra H^{i+1}_{p^r}(k,\,i)\lra H^{i+1}\bigl(X,\,\Z/p^r(i)\bigr)\lra H^0\bigl(k\,,\,H^{i+1}(\ov{X}\,,\,\Z/p^r(i))\bigr)\lra 0\;.
\end{equation}
On the other hand, since
\[
H^0_{Zar}\bigl(\ov{X}\,,\,\sH^{i+1}_{p^r}(i)\bigr)=H^{i+1}_{nr}\bigl(\ov{X},\,\Z/p^r(i)\bigr)=H^{i+1}\bigl(\ov{k},\,\Z/p^r(i)\bigr)=0
\]by Thm.\;\ref{3p2temp} (4), it follows from \eqref{3p6p1temp} that
\begin{equation}\label{3p9p7temp}
	H^{i+1}\bigl(\ov{X}\,,\,\Z/p^r(i)\bigr)=H^1_{Zar}\bigl(\ov{X}\,,\,\sH^i_{p^r}(i)\bigr)\,.
\end{equation}
For $i=0$, we can deduce from \eqref{3p9p7temp} and \eqref{3p2p2temp} that $H^1(\ov{X},\,\Z/p^r)=0$. Hence \eqref{3p9p6temp} yields $H^1(k,\,\Z/p^r)\cong H^1(X,\,\Z/p^r)$ in this case, recovering the result of Prop.\;\ref{3p4temp}.

For $i=1$, from \eqref{3p9p7temp} and \eqref{3p2p3temp} we get
\begin{equation}\label{3p9p8temp}
	H^2\bigl(\ov{X}\,,\,\Z/p^r(1)\bigr)\cong \CH^1(\ov{X})/p^r\;.
\end{equation}(This can also be deduced from Prop.\;\ref{3p7temp}.) So the case $i=1$ of \eqref{3p9p6temp} gives an exact sequence
\begin{equation}\label{3p9p9temp}
	0\lra H^2\bigl(k,\,\Z/p^r(1)\bigr)\lra H^2\bigl(X,\,\Z/p^r(1)\bigr)\lra H^0\bigl(k\,,\,\CH^1(\ov{X})/p^r\bigr)\lra 0\,.
\end{equation}

Let $\sK_i$ denote the Zariski sheaf defined by Quillen's $K$-theory. Using the Bloch--Kato--Gabber theorem, we can deduce an exact sequence (\cite[Thm.\;4.13]{GrosSuwa88b})
\begin{equation}\label{3p9p10temp}
	0\lra H^{i+1}_{Zar}\bigl(X\,,\,\sK_i\bigr)/p^r\lra H^{i-1}_{Zar}\bigl(X\,,\,\sH^i_{p^r}(i)\bigr)\lra \CH^i(X)[p^r]\lra 0\;
\end{equation}for $X$ and similarly for $\ov{X}$.

Now assume further that $X$ is a projective homogeneous variety. Then $\CH^i(\ov{X})$ is torsion-free and  by \cite[\S\,1, Prop.\;1]{Merkurjev95AlgiAna},
\[
H^{i-1}_{Zar}\bigl(\ov{X}\,,\,\sK_i\bigr)=K_1(\ov{k})\otimes\CH^{i-1}(\ov{X})\,.
\](As is well known, the Quillen $K$-group $K_1(\ov{k})$ here agrees with the Milnor $K$-group $K_1^M(\ov{k})$. But the Quillen $K$-theoretic viewpoint is a more natural way to understand the essentials in the proof.) So, in this case from \eqref{3p9p10temp} we get
\begin{equation}\label{3p9p11temp}
	H^{i-1}_{Zar}\bigl(\ov{X}\,,\,\sH^i_{p^r}(i)\bigr)\cong H^{i-1}_{Zar}\bigl(\ov{X}\,,\,\sK_i\bigr)/p^r\cong K_1(\ov{k})/p^r\otimes \CH^{i-1}(\ov{X})/p^r\;.
\end{equation}For $i=2$ we can combine \eqref{3p9p11temp} and \eqref{3p9p7temp} to get an isomorphism
\begin{equation}\label{3p9p12temp}
	H^3\bigl(\ov{X}\,,\,\Z/p^r(2)\bigr)\cong K_1(\ov{k})/p^r\otimes\CH^1(\ov{X})/p^r\;.
\end{equation}
In particular, this holds when $X$ is a smooth projective quadric of dimension $\ge 1$.

\medskip

We end this section with a few remarks on cohomology with divisible coefficients.

\medskip

\newpara\label{3p10temp} Given integers $b,\,i\in\N$, by taking direct limits we can define the functors
\[
\begin{split}
H^b\bigl(\cdot\,,\,\Q_p/\Z_p(i)\bigr):&=\ilim_rH^b\big(\cdot\,,\,\Z/p^r(i)\bigr)\,,\\
H^b_{nr}\bigl(\cdot,\,\Q_p/\Z_p(i)\bigr):&=\ilim_rH^b_{nr}\bigl(\cdot,\,\Z/p^r(i)\bigr)\,.
\end{split}
\]It is easy to extend all the previous discussions in this section to these cohomology groups with divisible coefficients. In particular, for a smooth connected $k$-variety $X$, by taking the limits of \eqref{3p3p1temp} and \eqref{3p6p2temp} we obtain a natural map
\[
\eta^j_{\infty}\,:\; H^j\bigl(k,\,\Q_p/\Z_p(j-1)\bigr)\lra H^j_{nr}\bigl(X,\,\Q_p/\Z_p(j-1)\bigr)\;
\]for each $j\ge 1$, and a cycle class map with divisible coefficients
\[
\cl^i_{\infty}\,:\;\CH^i(X)\otimes(\Q_p/\Z_p)\lra H^{2i}\bigl(X,\,\Q_p/\Z_p(i)\bigr)
\]for each $i\ge 1$.

A useful standard fact (which follows from \cite[Lemma\;6.6]{Izhboldin91}) is that the sequence
\begin{equation}\label{3p10p1temp}
  0\to H^j_{p^r}(K,\,j-1)\to H^j\bigl(K,\,\Q_p/\Z_p(j-1)\bigr)\xrightarrow[]{\times p^r}H^j\bigl(K,\,\Q_p/\Z_p(j-1)\bigr)\to 0
\end{equation} is exact for any field extension $K/k$. As a consequence, we have an identification
\begin{equation}\label{3p10p2temp}
  \ker(\eta^j_r)=\ker(\eta^j_{\infty})[p^r]\,.
\end{equation}

\section{Some general observations for quadrics}

From now on, we work over a field $F$ of characteristic 2 (although this characteristic restriction is unnecessary in some results, e.g., Prop.\;\ref{4p1temp}, Cor.\;\ref{4p2temp} and Prop.\;\ref{4p4temp}).

Let $X$ denote a smooth projective quadric of dimension $\ge 1$ over $F$.  Notation and results in \S\,\ref{sec3} will be applied with $p=2$.

For each $j\ge 1$, we have the natural maps (cf. \eqref{3p3p1temp} and \eqref{3p10temp})
\[
\eta^j_r=\eta^j_{r,\,X}\,:\; H^j\bigl(F,\,\Z/2^r(j-1)\bigr)\lra H^j_{nr}\bigl(X,\,\Z/2^r(j-1)\bigr)\;,\quad r\ge 1
\]and
\[
\eta^j_\infty=\eta^j_{\infty,\,X}\,:\; H^j\bigl(F,\,\Q_2/\Z_2(j-1)\bigr)\lra H^j_{nr}\bigl(X,\,\Q_2/\Z_2(j-1)\bigr)\;.
\]

\begin{prop}[{\cite[Prop.\;2.5]{KahnRostSujatha98}}]\label{4p1temp}
  With notation as above, the following statements hold:

  \begin{enumerate}
    \item If $X$ is isotropic over $F$, then the maps $\eta^j_{r,\,X}$ and $\eta^j_{\infty,\,X}$ are all isomorphisms.
    \item In general, the maps $\eta^j_{r,\,X}$ and $\eta^j_{\infty,\,X}$ all have $2$-torsion kernel and cokernel.
    \item Let $Y$ be another smooth projective quadric of dimension $\ge 1$ over $F$. If $X$ is isotropic over the function field $F(Y)$, then there is a natural commutative diagram
        \[
        \xymatrix{
        & H^j\bigl(F,\,\Z/2^r(j-1)\bigr) \ar[dl]_{\eta^j_{r,\,X}} \ar[dr]^{\eta^j_{r,\,Y}} & \\
        H^j_{nr}\bigl(X,\,\Z/2^r(j-1)\bigr) \ar[rr]_{\rho} &&  H^j_{nr}\bigl(Y,\,\Z/2^r(j-1)\bigr)
        }
        \]
        If moreover $Y$ is isotropic over $F(X)$, the map $\rho$ in the above diagram is an isomorphism.

        Similar results hold in the case of divisible coefficients.
  \end{enumerate}
\end{prop}

Note that Prop.\;\ref{4p1temp} is characteristic free, because the proof is more or less a formal consequence of some general facts from the algebraic and geometric theories of quadratic forms. For example, part (1) holds simply because a smooth quadric with a rational point is a rational variety, and part (2) is immediate from the functoriality and the easy fact that any anisotropic quadratic form becomes isotropic over a quadratic extension.

\begin{coro}[{\cite[Lemma\;1.3]{ColliotSujatha95}}]\label{4p2temp}
	Suppose $\vp$ is a Pfister form of dimension $\ge 4$ and let $\psi$ be a neighbor of $\vp$. Let $X$ and $Y$ be the projective quadrics defined by $\vp$ and $\psi$ respectively.
	
	Then for every $j\ge 1$ we have isomorphisms
	\[
	\ker(\eta^j_{r,\,X})=\ker(\eta^j_{r,\,Y})\quad\text{and}\quad \coker(\eta^j_{r,\,X})\cong\coker(\eta^j_{r,\,Y})\,,\quad r\ge 1\;,
	\]and similarly in the case of divisible coefficients.
\end{coro}
\begin{proof}
	Apply Prop.\;\ref{4p1temp} (3).
\end{proof}

The lemma below relies on the characteristic 2 assumption, because in the proof we need the surjectivity part of the exact sequence \eqref{3p10p1temp}. In fact, this is also a source of some differences between our results and the known results in characteristic $\neq 2$. (See e.g. the proof of Prop.\;\ref{5p3temp} below.)

\begin{lemma}\label{4p3temp}
  For every  $r\ge 1$ and $j\ge 1$, $\ker(\eta^j_r)=\ker(\eta^j_{\infty})$ and there is an exact sequence
  \[
  0\lra \ker(\eta^j_r)\lra \coker(\eta^j_r)\lra \coker(\eta^j_{\infty})\lra 0\,.
  \]
\end{lemma}
\begin{proof}
The first assertion is immediate from \eqref{3p10p2temp} and Prop.\;\ref{4p1temp} (2).

By functoriality and \eqref{3p10p1temp} we have the commutative diagram with exact rows
{\small
\[
\xymatrix{
0 \ar[r] & H^j\bigl(F,\,\Z/2^r(j-1)\bigr) \ar[d]_{\eta^j_r} \ar[r] & H^j\bigl(F,\,\Q_2/\Z_2(j-1)\bigr) \ar[d]_{\eta^j_\infty} \ar[r]^{\times 2^r} & H^j\bigl(F,\,\Q_2/\Z_2(j-1)\bigr) \ar[r] \ar[d]_{\eta^j_\infty} & 0\\
0 \ar[r] & H_{nr}^j\bigl(X,\,\Z/2^r(j-1)\bigr) \ar[r] & H^j_{nr}\bigl(X,\,\Q_2/\Z_2(j-1)\bigr) \ar[r]^{\times 2^r} & H^j_{nr}\bigl(X,\,\Q_2/\Z_2(j-1)\bigr) &
}
\]
}Applying the snake lemma to this diagram yields the desired exact sequence, noticing that $\coker(\eta^j_{\infty})$ is 2-torsion.
\end{proof}

Thanks to the above lemma, when studying the kernel and the cokernel of $\eta^j_r$ we may restrict to the case $r=1$.

\medskip

The following result is a special case of \cite[Prop.\;5.2]{Kahn99MotivicCellular}.

\begin{prop}\label{4p4temp}
  There is a natural isomorphism $\coker(\eta^3_{\infty})\cong \ker(\cl^2_{\infty})$ for the smooth projective quadric $X$.
\end{prop}

\section{Results for conic curves}\label{sec5}

To simplify the notation, we henceforth write
\[
\begin{split}
  H^b(\cdot\,,\,i)&=H^b_{\et}\big(\cdot\,,\,\Z/2(i)\big)\,,\;\sH^b(i)=\sH^b_{2}(i)\,,\\
H^{i+1}(\cdot)&=H^{i+1}(\cdot\,,\,i)\,,\; H^{i+1}_{nr}(\cdot)=H^{i+1}_{nr}(\cdot\,,\,i)\;.
\end{split}
\]
In this section and the next, we study the maps
\[
\begin{split}
\eta^j:=\eta^j_1\,:\; H^j(F)&\lra H^j_{nr}(X)\,,\\
\eta^j_{\infty}\,:\; H^j\bigl(F,\,\Q_2/\Z_2(j-1))&\lra H^j_{nr}\bigl(X\,,\,\Q_2/\Z_2(j-1)\bigr)
\end{split}
\]for the quadric $X$. They are both isomorphisms if $X$ is isotropic (Prop.\;\ref{4p1temp} (1)) or $j=1$ (Prop.\;\ref{3p4temp}). So we assume $X$ is anisotropic and $j\ge 2$ in the sequel.

\

For a conic curve an explicit description is known for  the kernel of the map $\eta^j$.

\begin{prop}\label{5p1temp}
	Let $X\subseteq\bP^2_F$ be the conic associated to a quaternion $F$-algebra $D$.
	
	Then for all $j\ge 2$,
	\[
	\ker\left(\eta^j\,:\; H^j(F)\lra H^j_{nr}(X)\right)=\{(D)\cup\xi\,|\,\xi\in K^M_{j-2}(F)\}
	\]where $(D)\in H^2(F)=\Br(F)[2]$ denotes the Brauer class of $D$.
\end{prop}
\begin{proof}
	Since $H^j_{nr}(X)=H^0_{Zar}\big(X,\,\sH^j(j-1)\big)$ is a subgroup of $H^j(F(X))$, $\ker(\eta^j)$ coincides with the kernel of the natural map $H^j(F)\to H^j(F(X))$. The result thus follows from \cite[Thm.\;3.6]{AravireJacob09Contemp493}.
\end{proof}

\begin{remark}\label{5p2temp}
	In Prop.\;\ref{5p1temp}, the case $j=2$ amounts to an exact sequence
	\[
	0\lra \Z/2\overset{\delta}{\lra} H^2(F)\lra H^2(F(X))
	\]where the map $\delta$ sends $1$ to the Brauer class $(D)$. This is in fact a special case of Amitsur's theorem (cf. \cite[Thm.\;5.4.1]{GilleSzamuely17}).
	
	For $j=3$, the proposition gives a characteristic 2 analogue of a theorem of Arason \cite[Satz\;5.4]{Arason75JA}, i.e., we have an exact sequence
	\[
	F^*\xrightarrow[]{\cup (D)} H^3(F)\lra H^3(F(X))\;.
	\]The  kernel of the first map is the group $\mathrm{Nrd}(D^*)$ of reduced norms of $D$ by \cite[p.94, Thm.\;6]{Gille00Ktheory}. So in this case we have an isomorphism $F^*/\Nrd(D^*)\cong \ker(\eta^3)$.
\end{remark}

The following result is slightly different from its counterpart in characteristic $\neq 2$ (cf. \cite[p.246, Remarks\;(4)]{Kahn95ArchMath}). As in Lemma\;\ref{4p3temp}, the characteristic 2 assumption is crucial since the surjectivity in \eqref{3p10p1temp} is needed.

\begin{prop}\label{5p3temp}
	Let $X\subseteq\bP^2_F$ be a smooth anisotropic conic. Then $\coker(\eta^2)\cong\Z/2$.
\end{prop}
\begin{proof}
		By  \eqref{3p2p5temp}, we can identify $\eta^2$ with the natural map $\Br(F)[2]\to \Br(X)[2]$. By considering the
	 the Hochschild--Serre spectral sequence for the sheaf $\mathbb{G}_m$ on $X$ we can get an exact sequence
	\[
	0\lra \Z/2\lra \Br(F)\lra \Br(X)\lra 0\;.
	\]The Brauer group $\Br(F)$ is 2-divisible since $\car(F)=2$ (cf. \eqref{3p10p1temp}). Thus, multiplying the above sequence by 2 and applying the snake lemma shows $\coker(\eta^2)\cong \Z/2$ as claimed.
\end{proof}

Our goal now is to extend Peyre's results in \cite[\S\,2]{Peyre95ProcSymPureMath58} to characteristic 2. A common feature in the arguments of Peyre's and ours is the use of vanishing results for terms in the Hochschild--Serre spectral sequence. But the spectral sequences in the two different cases look remarkably different, because vanishing holds for different reasons and hence the positions of vanishing terms are not same. (In characteristic $\neq 2$, homotopy invariance is an ingredient guaranteeing some vanishing results.) Indeed, the cohomology groups of the projective line are already different in different characteristics (see \eqref{5p4temp} below).  As was mentioned in \eqref{3p9temp}, in our situation the cohomological 2-dimension of the base field plays a key role, and the description of certain cohomology groups relies on the vanishing result in \eqref{3p2p2temp} (or its consequence \eqref{3p2p4temp}).

\

\newpara\label{5p4temp} Let $X\subseteq\bP^2_F$ be a smooth conic. Fix $i\in\N$. In the Hochschild--Serre spectral sequence
$H^p(F,\,H^q(\ov{X},\,i))\Rightarrow H^{p+q}(X,\,i)$ we have
\[
	H^q(\ov{X},\,i)=H^q\big(\bP^1_{\ov{F}}\,,\,i\big)=\begin{cases}
		H^{i-1}(\ov{F},\,i-1)=K_{i-1}^M(\ov{F})/2\quad & \text{ if }\; q=i+1\,,\\
		H^i(\ov{F},\,i)=K_i^M(\ov{F})/2 \quad & \text{ if }\; q=i\,,\\
		0 \quad & \text{ otherwise }
	\end{cases}
\]
by \cite[(2.1.15)]{Gros85} (and the Bloch--Kato--Gabber theorem). This combined with \eqref{3p9p2temp}, \eqref{3p9p3temp}, \eqref{3p9p5temp} and \eqref{3p9p6temp} yields isomorphisms
\[\begin{split}
	H^i(F,\,i)&=H^0(F,\,H^i(\ov{X},\,i))\cong H^i(X,\,i)\,,\\
 H^i(F)&=H^1(F,\,H^{i+1}(\ov{X}))\cong H^{i+2}(X,\,i)\;
\end{split}\]
and an exact sequence
\begin{equation}\label{5p4p1temp}
	0\lra H^{i+1}(F)\lra H^{i+1}(X)\overset{\rho}{\lra} H^{i-1}(F,\,i-1)\lra 0\;.
\end{equation}
Since $\dim X=1$, \eqref{3p6p1temp} gives a short exact sequence
\begin{equation}\label{5p4p2temp}
	0\lra H^1_{Zar}\big(X,\,\sH^i(i)\big)\overset{\iota}{\lra} H^{i+1}(X)\lra H^{i+1}_{nr}(X)\lra 0\,.
\end{equation}
Also, by the Gersten resolution of $\sH^i(i)$ we have
\[
	H^1_{Zar}\big(X,\,\sH^i(i)\big)=\coker\biggl(H^i(F(X),\,i)\lra \bigoplus_{x\in X^{(1)}}H^{i-1}(\kappa(x)\,,\,i-1)\biggr)\;.
\]
In particular, there is a natural surjection
\[
\bigoplus_{x\in X^{(1)}}H^{i-1}(\kappa(x)\,,\,i-1)\lra H^1_{Zar}\big(X\,,\,\sH^i(i)\big)\;.
\]
The same arguments as in the proof of \cite[p.377, Lemma\;2.1]{Peyre95ProcSymPureMath58} show that the diagram
\begin{equation}\label{5p4p3temp}
	\xymatrix{
		\bigoplus\limits_{x\in X^{(1)}}H^{i-1}(\kappa(x)\,,\,i-1) \ar[d]_{\oplus\Cor_{\kappa(x)/F}} \ar@{->>}[rr] && H^1_{Zar}\big(X\,,\,\sH^i(i)\big) \ar[d]^{\iota} \\
		H^{i-1}(F,\,i-1) && \ar[ll]_{\rho} H^{i+1}(X)
	}
\end{equation}
is commutative. Let $D$ be the quaternion $F$-algebra corresponding to the conic $X$. For each $x\in X^{(1)}$ we have $(D)_{\kappa(x)}=0$ in $H^2(\kappa(x))=\Br(\kappa(x))[2]$. Therefore,
\[
\Cor_{\kappa(x)/F}(\al)\cup (D)=\Cor_{\kappa(x)/F}\big(\al\cup (D)_{\kappa(x)}\big)=0\quad\text{ for all }\,\al\in H^{i-1}(\kappa(x)\,,\,i-1)\,.
\]This together with the commutative diagram \eqref{5p4p3temp} shows that the composite map
\[
H^1_{Zar}(X,\,\sH^i(i))\overset{\iota}{\lra} H^{i+1}(X)\overset{\rho}{\lra}H^{i-1}(F,\,i-1)\xrightarrow[]{\cup (D)} H^{i+1}(F)
\]
is 0. Now, using \eqref{5p4p2temp} we can define
\begin{equation}\label{5p4p4temp}
	N\;:\; H^{i+1}_{nr}(X)\lra H^{i+1}(F)
\end{equation}to be the unique homomorphism making the following diagram commute:
\[
\xymatrix{
	H^{i+1}(X)
	\ar@{->>}[d]_{\rho} \ar@{->>}[r] & H^{i+1}_{nr}(X) \ar[d]^{N} \\
	H^{i-1}(F,\,i-1) \ar[r]^{\quad\cup (D)}&  H^{i+1}(F).
}
\]
Note that \eqref{5p4p2temp} also yields a homomorphism
\[
	\tau\;:\; \ker(\eta^{i+1})=\ker\big(H^{i+1}(F)\to H^{i+1}(F(X))\big)\lra H^1_{Zar}\big(X\,,\,\sH^i(i)\big)
\]
such that the diagram
\[
\xymatrix{
	0\ar[r] & \ker(\eta^{i+1}) \ar[d]_{\tau} \ar[r] & H^{i+1}(F)  \ar[d] \ar[r]^{\eta^{i+1}} & H^{i+1}_{nr}(X) \ar@{=}[d] & \\
	0 \ar[r] & H^1_{Zar}\big(X,\,\sH^i(i)\big) \ar[r]^{\qquad\iota} & H^{i+1}(X)   \ar[r]&  H^{i+1}_{nr}(X) \ar[r] & 0
}
\]is commutative with exact rows.

Now we have the following complex
\begin{equation}\label{5p4p5temp}
	\begin{split}
		0&\lra \ker(\eta^{i+1})\overset{\tau}{\lra} H^1_{Zar}\big(X,\,\sH^i(i)\big)\overset{\rho\circ\iota}{\lra} H^{i-1}(F,\,i-1)\\
		&\xrightarrow[]{\cup (D)} H^{i+1}(F)\xrightarrow[]{\eta^{i+1}} H^{i+1}_{nr}(X) \overset{N'}{\lra} H^{i-1}(F,\,i-1)\cup (D)\lra 0
	\end{split}
\end{equation}
where
\[
H^{i-1}(F,\,i-1)\cup (D):=\im\left( H^{i-1}(F,\,i-1)\xrightarrow[]{\cup (D)}   H^{i+1}(F)\right)
\]and $N'$ is induced by the map $N$ in \eqref{5p4p4temp}.

\

We can now prove a characteristic 2 counterpart of \cite[p.379, Prop.\;2.2]{Peyre95ProcSymPureMath58}.

\begin{prop}\label{5p5temp}
	The complex $\eqref{5p4p5temp}$ is exact except possibly at the third term $H^{i-1}(F,\,i-1)$ and the fifth term $H^{i+1}_{nr}(X)$.
	
	Moreover,
	\[
	\frac{\ker(N')}{\im(\eta^{i+1})}\cong  \frac{\ker(\cup (D))}{\im(\rho\circ\iota)}=\frac{\ker(\cup (D))}{\im\biggl(\oplus\Cor_{\kappa(x)/F}\,:\;
		\bigoplus\limits_{x\in X^{(1)}}H^{i-1}(\kappa(x),\,i-1)\to H^{i-1}(F,\,i-1)\biggr)}
	\]	
	If $i\le 2$, then $\eqref{5p4p5temp}$ is exact everywhere.
\end{prop}
\begin{proof}
	The injectivity of $\tau$ is a consequence of the injectivity of $H^{i+1}(F)\to H^{i+1}(X)$, and the surjectivity of $N'$ follows from that of the map $\rho$ (cf. \eqref{5p4p1temp}).
	The equality $\im(\tau)=\ker(\rho\circ\iota)$ can be easily shown by a diagram chase, and the equality $\im(\cup (D))=\ker(\eta^{i+1})$ was Prop.\;\ref{5p1temp}.
	
	To get the isomorphism $\frac{\ker(N')}{\im(\eta^{i+1})}\cong  \frac{\ker(\cup (D))}{\im(\rho\circ\iota)}$, it suffices to apply the snake lemma to the following commutative diagram
	\[
	\xymatrix{
		0 \ar[r] & H^1_{Zar}\big(X,\,\sH^i(i)\big) \ar[d]_{(\rho\circ\iota)'} \ar[r]^{\quad\iota} & H^{i+1}(X) \ar@{->>}[d]^{\rho}
		\ar[r] & H^{i+1}_{nr}(X) \ar@{->>}[d]^{N'} \ar[r] & 0\\
		0 \ar[r] & \ker(\cup(D)) \ar[r] & H^{i-1}(F,\,i-1) \ar[r]^{\cup (D)\quad} & H^{i-1}(F,\,i-1)\cup (D) \ar[r] & 0.
	}
	\]	
	From the diagram \eqref{5p4p3temp} we find
	\[
	\im(\rho\circ\iota)=\im\biggl(\oplus\Cor_{\kappa(x)/F}\,:\;
	\bigoplus\limits_{x\in X^{(1)}}H^{i-1}(\kappa(x),\,i-1)\to H^{i-1}(F,\,i-1)\biggr)\;.
	\]
	
	If $i=0$, then $H^{i-1}(F,\,i-1)=0$. If $i=1$, then
	\[
	\im(\rho\circ\iota)=\im(\deg\,:\;\CH^1(X)/2\lra \Z/2)
	\]and it is equal to $\ker\bigl(\Z/2\xrightarrow[]{\cup (D)}H^2(F)\bigr)$. If $i=2$, the equality $\im(\rho\circ\iota)=\ker(\cup (D))$ follows from \cite[p.94, Thm.\;6]{Gille00Ktheory}.
\end{proof}

\begin{coro}\label{5p6temp}
	Let $X\subseteq\bP^2_F$ be a smooth conic with associated quaternion algebra $D$. Then we have isomorphisms
	\[
	\coker(\eta^3)\cong H^1(F,\,1)\cup (D)\cong F^*/\Nrd(D^*)\cong \ker(\eta^3)\;.
	\]
\end{coro}
\begin{proof}
The first isomorphism follows from the $i=2$ case of Prop.\;\ref{5p5temp}. The other isomorphisms have been discussed in Remark\;\ref{5p2temp}.
\end{proof}

\begin{coro}\label{5p7temp}
	Suppose $\vp$ is  the reduced norm form of a quaternion division algebra $D$.   Then  we have isomorphisms
	\[
	\begin{split}
		\ker(\eta^j)&=H^{j-2}(F,\,j-2)\cup (D)\;,\quad \text{for all } j\ge 2\;,\\
		\ker(\eta^2)&\cong \coker(\eta^2)\cong 			\Z/2\;,\\
		\ker(\eta^3)&\cong\coker(\eta^3)\cong H^1(F,\,1)\cup (D)\cong F^*/\Nrd(D^*)\;.
	\end{split}
	\]
\end{coro}
\begin{proof}
	This follows from Cor.\;\ref{4p2temp} and the corresponding results for conics (Props.\;\ref{5p2temp}, \ref{5p3temp} and Cor.\;\ref{5p6temp}).
\end{proof}

\begin{remark}\label{5p8temp}
  Suppose $3\le\dim\vp\le 4$. Then the cycle class map with divisible coefficients
\[
\cl^2_{\infty}\,:\; \CH^2(X)\otimes(\Q_2/\Z_2)\lra H^4\bigl(X,\,\Q_2/\Z_2(2)\bigr)\;
\]is injective and the map  $\eta^3_{\infty}: H^3\bigl(F,\,\Q_2/\Z_2(2)\bigr)\to H^3_{nr}\bigl(X,\,\Q_2/\Z_2(2)\bigr)$ is surjective.

By Prop.\;\ref{4p4temp}, if one of the two groups $\ker(\cl^2_{\infty})$  and $\coker(\eta^3_{\infty})$ is trivial, then so is the other. If $\dim\vp=3$, then $\dim X=1$ and $\CH^2(X)=0$, so trivially $\ker(\cl^2_{\infty})=0$. If $\dim\vp=4$, as in the proof of Cor.\;\ref{5p7temp} we may reduce to the case where $\vp$ is a 2-Pfister form. Then we can apply Cor.\;\ref{4p2temp} to obtain $\coker(\eta^3_{\infty})=0$ by passing to the case of conics.

As a corollary, we have an isomorphism $\ker(\eta^3)\cong \coker(\eta^3)$ in view of Lemma\;\ref{4p3temp}. In fact, it is shown in \cite[Thm.\;5.6]{HLS21} that
\[
\ker(\eta^3)=\{0\}\cup\{(a)\cup (b)\cup (c]\,|\,a,\,b,\,c\in F^*,\, \vp \text{ is similar to a subform of } \langle\!\langle a,\,b\,;\,c]]\}\,.
\]
\end{remark}

\section{Quadrics of dimension $\ge 2$}

In this section, we prove our main results about low degree unramified cohomology for quadrics of dimension at least 2. These extend the main theorems of \cite{Kahn95ArchMath} to characteristic 2. While the basic strategy is derived from Kahn's paper, we do need to check quite a few details, and in doing so (e.g. in the proofs of Lemma\;\ref{6p3temp} (2) and Prop.\;\ref{6p4temp}) we do use some ingredients (such as the exact sequences \eqref{3p9p3temp} and \eqref{3p9p6temp}) that do not show up in characteristic $\neq 2$. In the proof of Thm.\;\ref{6p11temp} we even have to use a different approach which builds upon Lemma\;\ref{4p3temp}, a result that is special in characteristic 2 itself.

Throughout this section, let $\vp$ be a nondegenerate  quadratic form of dimension $\ge 3$ over the field $F$ and let $X$ be the projective quadric defined by $\vp$.

\begin{lemma}\label{6p1temp}
	For each $i\in\N$, let $\xi^i:\CH^i(X)/2\to H^0(F,\,\CH^i(\ov{X})/2)$ be the canonical map. Assume $\vp$ is anisotropic.
	
	\begin{enumerate}
		\item If  $\dim\vp=4$ and $e_1(\vp)=0$, then  $\ker(\xi^1)\cong\coker(\xi^1)\cong\Z/2$.
		
		Otherwise $\xi^1$ is an isomorphism.
		\item Suppose $4\le \dim\vp\le 5$. Then $\xi^2=0$, $\coker(\xi^2)\cong \Z/2$ and
		\[
		\ker(\xi^2)\cong\begin{cases}
			(\Z/2)^2\quad & \text{ if }\; \dim\vp=5 \text{ and $\vp$ is a Pfister neighbor}\,,\\
			\Z/2\quad & \text{ otherwise}\,.
		\end{cases}
		\]
		\item Suppose $\dim\vp\ge 6$.
		\begin{enumerate}
			\item If $\vp$ is an Albert form, then $\ker(\xi^2)\cong\coker(\xi^2)\cong \Z/2$.
			\item If $\vp$ is a $3$-Pfister neighbor, then $\ker(\xi^2)\cong     \Z/2$ and $\coker(\xi^2)=0$.
			\item In all the other cases, $\xi^2$ is an isomorphism.
		\end{enumerate}
	\end{enumerate}
\end{lemma}
\begin{proof}
The idea of proof is to compute the maps $\xi^i$ explicitly, by using the Galois module structures of the Chow groups in question. Indeed, the results we need about Chow groups can be found in \cite{Karpenko90AlgGeoInv} and \cite{HLS21}.

For example, if  $\vp$ is a 3-Pfister neighbor of dimension $\ge 6$, then  $\CH^2(X)_{\tors}\cong \Z/2$ by \cite[Thm.\;5.3]{HLS21},
\[
\CH^2(X)=\CH^2(X)_{\tors}\oplus \Z.h^2\quad\text{ and }\quad
\CH^2(\ov{X})^{\Gal(\ov{F}/F)}=\Z.h^2\,.
\] The map $\xi^2$ is the identity on $(\Z/2\Z).h^2$ and 0 on $\CH^2(X)_{\tors}$.
\end{proof}

Recall that we have defined the cycle class maps  in \eqref{3p6p2temp}. Here we only need the mod 2 case.

\begin{coro}\label{6p2temp}
	The cycle class map $\cl^2_X: \CH^2(X)/2\to H^4(X,\,2)$ is injective in the following cases:
	
	\begin{enumerate}
		\item $\dim\vp=6$ and $\vp$ is neither an Albert form nor a Pfister neighbor.
		\item $7\le \dim\vp\le 8$ and $\vp$ is not a Pfister neighbor.
		\item $\dim\vp>8$.
	\end{enumerate}
\end{coro}
\begin{proof}
By functoriality, the result follows from  Lemma\;\ref{6p1temp} (3) and the injectivity of  the cycle class map $\cl^2_{\ov{X}}$ for $\ov{X}$ (Prop.\;\ref{3p8temp}).
\end{proof}

\begin{lemma}\label{6p3temp}
	Suppose $X$ is defined by the reduced norm form of a quaternion division algebra $D$ over $F$.

\begin{enumerate}
  \item  Let $P$ be a closed point in $X$. The cycle class map $\cl^2_X:\CH^2(X)/2\to H^4(X,\,2)$ sends the class $[P]\in\CH^2(X)/2$ to the cup product $(D)\cup \cl^1_X(h)\in H^4(X,\,2)$.
  \item We have natural isomorphisms
	\[
	H^2(F\,,\,H^2(\ov{X},\,1)\otimes \Z/2(1))\cong H^2(F)\oplus H^2(F)\cong H^1\bigl(F\,,\,H^3(\ov{X},\,2)\bigr)\,.
	\]
\end{enumerate}
\end{lemma}
\begin{proof}
(1) This is a restatement of \cite[Lemma\;2]{Kahn95ArchMath}, whose proof still works in characteristic 2. In fact, the key ingredients in that proof all have characteristic 2 versions: If $C\subseteq X$ is a smooth hyperplane section, then $H^2(F)=\ker\bigl(H^2(C)\to H^2(\ov{C})\bigr)$ by \eqref{3p9p6temp}, the map $\cl^1_C$ is injective by Prop.\;\ref{3p7temp}, and functorial properties of Gysin maps hold by \cite[Chap.\;II, Cor.\;2.2.5 and Prop.\;4.2.3]{Gros85}.

(2) We have seen in \eqref{3p9p8temp} that $H^2(\ov{X},\,1)\cong \CH^1(\ov{X})/2$. The Galois action on $\CH^1(\ov{X})/2$ being trivial, the first isomorphism follows.
	
	To obtain the second isomorphism, note that
	\[
	H^3(\ov{X},\,2)\cong K_1(\ov{F})/2\otimes\CH^1(\ov{X})/2\cong (K_1(\ov{F})/2)^2
	\]by \eqref{3p9p12temp}. Applying the second formula in \eqref{3p9p5temp} with $i=1$, we get
	\[
	H^1(F\,,\,H^3(\ov{X},\,2))\cong H^1(F,\,K_1(\ov{F})/2)^2\cong H^2(F)\oplus H^2(F)\,.
	\]This completes the proof.
\end{proof}

Note that the proof of Lemma\;\ref{6p3temp} (2) used a special case of the second isomorphism in \eqref{3p9p5temp}, which is only valid in characteristic 2. This result is relied on in the proof of Prop.\;\ref{6p4temp} below, and together with \eqref{3p9p3temp} this serves as a characteristic 2 substitute for Shyevski's computation in \cite[\S\,5]{Shyevski90}, which was used in \cite{Kahn95ArchMath} in characteristic $\neq 2$.

\begin{prop}\label{6p4temp}
	The cycle class map $\cl^2_X: \CH^2(X)/2\to H^4(X,\,2)$ is injective in the following two cases:
	
	\begin{enumerate}
		\item $\dim\vp=4$.
		\item $\dim\vp=5$ and $\vp$ is not a Pfister neighbor.
	\end{enumerate}
\end{prop}
\begin{proof}
As in the proof of \cite[Prop.\;1]{Kahn95ArchMath}, it is sufficient to treat the case where $\vp$ is the reduced norm form of a quaternion division algebra $D$ over $F$.

By Lemma\;\ref{6p3temp} (1), the generator of $\CH^2(X)/2\cong \Z/2$ is mapped to the cup product $(D)\cup \cl^1_X(h)\in H^4(X,\,2)$ by the cycle class map $\cl^2_X$. So it suffices to show $(D)\cup \cl^1_X(h)\neq 0$ in $H^4(X,\,2)$.
	
	Taking $j=3$ and $i=2$ in \eqref{3p9p3temp} we obtain an exact sequence
	\[
	0\lra H^1(F,\,H^3(\ov{X},\,2))\lra H^4(X,\,2)\lra H^0(F,\,H^4(\ov{X},\,2))\lra 0\,.
	\]Thus,
	\[
	(D)\cup\cl^1_X(h)\in \ker\bigl(H^4(X,\,2)\lra H^4(\ov{X},\,2)\bigr)=H^1(F,\,H^3(\ov{X},\,2))\,.
	\]
	By Lemma\;\ref{6p3temp} (2), we have $H^1(F,\,H^3(\ov{X},\,2))=H^2(F,\,H^2(\ov{X})\otimes \Z/2(1))$. Therefore, we need only to show $(D)\cup \cl^1_X(h)$ is nonzero in $H^2(F,\,H^2(\ov{X})\otimes \Z/2(1))$.
	
	Consider the commutative diagram
	\[
	\begin{CD}
		\CH^1(X)/2 @>>> \CH^1(\ov{X})/2\\
		@V{\cl^1_X}VV @VV{\cl^1_{\ov{X}} }V\\
		H^2(X,\,1) @>>> H^2(\ov{X},\,1)
	\end{CD}
	\]The top horizontal map can be identified with the map
	\[
	(0,\,\id)\;:\; (2\Z/4\Z).\ell_1\oplus (\Z/2).h\lra (\Z/2).\ell_1\oplus (\Z/2).h\;.
	\]
	The map $\cl^1_{\ov{X}}$ being an isomorphism, it follows that $\cl^1_{X}(h)\neq 0$ even in $H^2(\ov{X})$. Now the map
	\[
	\Z/2\lra H^2(\ov{X})\;;\quad 1\longmapsto \cl^1_X(h)
	\]can be identified with the map
	\[
	\iota\;:\; \Z/2\lra \Z/2\oplus\Z/2=\CH^1(\ov{X})/2\;;\quad 1\longmapsto (0,\,1)=h\;.
	\]We have the commutative diagram
	\begin{equation}\label{6p4p1temp}
		\begin{CD}
			\Z/2=H^0(F,\,0) @>{1\mapsto h}>> H^0(F,\,\CH^1(\ov{X})/2)\\
			@V{(D)\cup \cdot}VV @VV{(D)\cup \cl^1_X(\cdot)}V\\
			H^2(F,\,1) @>{\iota_*}>> H^2(F,\,H^2(\ov{X})\otimes \Z/2(1))
		\end{CD}
	\end{equation}
	The left vertical map in \eqref{6p4p1temp} is injective since $(D)\neq 0$ in $H^2(F)$. The bottom horizontal map $\iota_*$ in \eqref{6p4p1temp} can be identified with the map
	\[
	H^2(F)\lra H^2(F)\oplus H^2(F)\;;\quad \al\longmapsto (0,\,\al)
	\]via the first isomorphism given in Lemma\;\ref{6p3temp} (2). Hence $\iota_*$ is injective. Thus, the diagram \eqref{6p4p1temp} implies that $(D)\cup \cl^1_X(h)\neq 0$ in
	$H^2(F,\,H^2(\ov{X})\otimes \Z/2(1))$. As we have said before, this completes the proof.
\end{proof}

The statement of the following lemma resembles \cite[Lemma\;3]{Kahn95ArchMath}, but in its proof we need results that are special in characteristic 2.

\begin{lemma}\label{6p5temp}
\

	\begin{enumerate}
		\item If either $\dim\vp>4$, or $\dim\vp=4$ and $e_1(\vp)\neq 0$, then the map
		\[
		\cl^1=\nu^{1,\,1}\;:\; \CH^1(X)/2=H^0(F,\,0)\otimes \CH^1(X)/2\lra H^2(X)
		\]induces an isomorphism
		\[
		H^2(F)\oplus\bigl(\CH^1(X)/2\bigr)\simto H^2(X)\;.
		\]
		\item If $\dim\vp>4$, then the map
		\[
		\nu^{1,\,2}\;:\; H^1(F,\,1)\otimes\bigl(\CH^1(X)/2\bigr)\lra H^3(X)
		\](defined in $\eqref{3p6p3temp}$) induces an isomorphism
		\[
		H^3(F)\oplus\bigl(H^1(F,\,1)\otimes \CH^1(X)/2\bigr)\simto H^3(X)\;.
		\]
	\end{enumerate}
\end{lemma}
\begin{proof}
We only prove (2) here since the method of proof for (1) is similar.

We use the  commutative diagram
	\begin{equation}\label{6p5p1temp}
		\xymatrix{
			&&   H^1(F,\,1)\otimes\CH^1(X)/2 \ar[d]_{\nu^{1,\,2}} \ar@{=}[r] & H^0\bigl(F,\,K_1(\ov{F})/2\bigr)\otimes\CH^1(X)/2 \ar[d]_{\delta} & \\
			0 \ar[r] & H^3(F) \ar[r] & H^3(X) \ar[r] & H^0\bigl(F\,,\,K_1(\ov{F})/2\otimes\CH^1(\ov{X})/2\bigr) \ar[r] & 0
		}
	\end{equation}Here the bottom row is obtained by combining \eqref{3p9p6temp} with \eqref{3p9p12temp}. The right vertical map $\delta$ is an isomorphism since
	$\CH^1(X)/2\cong H^0(F,\,\CH^1(\ov{X})/2)=\CH^1(\ov{X})/2\cong\Z/2$.
\end{proof}

The following theorem is the characteristic 2 version of \cite[Thm.\;1]{Kahn95ArchMath}. Notice however that unlike the case of characteristic different from 2, the maps $\mu^{0,\,j}$ and $\eta^j$ are different in our situation.

\begin{thm}\label{6p6temp}
	If $\dim\vp>4$, or $\dim\vp=4$ and $e_1(\vp)\neq 0$, then the map
	$\eta^2: H^2(F)\to H^2_{nr}(X)$ is an isomorphism.
\end{thm}
\begin{proof}
	By Prop.\;\ref{3p7temp}, we have a natural exact sequence
	\[
	0\lra \CH^1(X)/2\overset{\cl^1}{\lra}H^2(X)\lra H^2_{nr}(X)\lra 0\;.
	\]
	This sequence together with Lemma\;\ref{6p5temp} (1) proves the theorem immediately.
\end{proof}

\begin{coro}\label{6p7temp}
  We have $\coker(\eta^2_{\infty})=0$ for any smooth projective quadric of dimension $\ge 1$.
\end{coro}
\begin{proof}
  Combine Lemma\;\ref{4p3temp}, Prop.\;\ref{5p3temp}, Cor.\;\ref{5p7temp} and Thm.\;\ref{6p6temp}.
\end{proof}

\begin{thm}\label{6p8temp}
	If $\dim\vp>4$, then the map
\[
\mu^{1,\,2}: H^1(F,\,1)\otimes\bigl(\CH^1(X)/2\bigr)\to H^1_{Zar}\bigl(X\,,\,\sH^2(2)\bigr)
\] is injective and there are isomorphisms
	\[
	\coker(\mu^{1,\,2})\cong \ker(\eta^3)\quad\text{and}\quad \coker(\eta^3)\cong \ker(\cl^2_X)\;.
	\]
\end{thm}
\begin{proof}
	Note that the map $\nu^{1,\,2}$ factors through $\mu^{1,\,2}$ by construction (cf. \eqref{3p6p4temp}). So we have the following commutative diagram
	\begin{equation}\label{6p8p1temp}
		\xymatrix{
			0 \ar[r] & H^1(F,\,1)\otimes\bigl(\CH^1(X)/2\bigr) \ar[d]_{\mu^{1,\,2}} \ar@{=}[r] & H^1(F,\,1)\otimes\bigl(\CH^1(X)/2\bigr) \ar[r] \ar[d]_{\nu^{1,\,2}} & 0 \ar[r] \ar[d] &  0\\
			0 \ar[r] & H^1_{Zar}\bigl(X\,,\,\sH^2(2)\bigr) \ar[r] & H^3(X,\,2) \ar[r] & M \ar[r] & 0
		}
	\end{equation}
	where the second row is obtained from the exact sequence \eqref{3p8p2temp} and
\[
M:=\ker(H^3_{nr}(X,\,2)\to \CH^2(X)/2)\;.
\]
 Applying the snake lemma to the above diagram and using Lemma\;\ref{6p5temp} (2) (and \eqref{3p8p2temp}), we get an exact sequence
	\begin{equation}\label{6p8p2temp}
\begin{split}
  		0&\lra H^1(F,\,1)\otimes \bigl(\CH^1(X)/2\bigr)\xrightarrow[]{\mu^{1,\,2}} H^1_{Zar}\bigl(X\,,\,\sH^2(2)\bigr)\\
  &\lra H^3(F)\overset{\eta^3}{\lra}H^3_{nr}(X)\lra \ker(\cl^2_X)\lra 0\;.
\end{split}	\end{equation}
	The theorem follows immediately from the above sequence.
\end{proof}

\begin{remark}\label{6p9temp}
	About the map $\mu^{1,\,2}$, the cases not covered by Theorem\;\ref{6p8temp} are:
	
	(a) $\dim\vp=3$; (b) $\dim\vp=4$ and $e_1(\vp)=0$; (c) $\dim\vp=4$ and $e_1(\vp)\neq 0$.
	
	In all these cases, we can still use the diagram \eqref{6p8p1temp} to get an exact sequence
	\[
		0\lra \coker(\mu^{1,\,2})\lra \coker(\nu^{1,\,2})\lra \ker\bigl(H^3_{nr}(X)\to \CH^2(X)/2\bigr)\lra 0
	\]
	and an equality $\ker(\mu^{1,\,2})=\ker(\nu^{1,\,2})$.

To get further information about $\ker(\mu^{1,\,2})$, we need to analyze the right vertical map $\delta$ in \eqref{6p5p1temp}. If $\vp$ is isotropic, then $\CH^1(X)/2\cong\CH^1(\ov{X})/2$ and $\delta$ is an isomorphism. So in this case the exact sequence \eqref{6p8p2temp} remains valid, and we get the same conclusions as in Theorem\;\ref{6p8temp}.
	
	Now we assume $\vp$ is anisotropic. Then we claim
\[
\begin{cases}
 \ker(\mu^{1,\,2})=\ker(\nu^{1,\,2})\cong \Nrd(D^*)/F^{*2}\;,\quad & \text{ in Cases (a) and (b)},\\
  \ker(\mu^{1,\,2})=\ker(\nu^{1,\,2})=0\;,\quad & \text{ in Case (c)}.
\end{cases}
\]
	
	In Case (a), the map $\delta$ in \eqref{6p5p1temp} can be identified with the zero map from $H^1(F,\,1)$ to itself, since the map
	$\CH^1(X)/2\to \CH^1(\ov{X})/2$ is the zero map from $2\Z/4\Z$ to $\Z/2\Z$. Thus, \eqref{6p5p1temp} yields an exact sequence
	\[
		0\lra \ker(\nu^{1,\,2})\lra H^1(F,\,1)\xrightarrow[]{\cup (D)} H^3(F)\lra \coker(\nu^{1,\,2})\lra H^1(F,\,1)\lra 0\,.
	\]	In Case (b), we can get an exact sequence of the same form, because the map $\delta$ can be viewed as the map
	\[
	(0\,,\,\id)\;:\; H^1(F,\,1)\oplus H^1(F,\,1)\lra H^1(F,\,1)\oplus H^1(F,\,1)\;.
	\]This proves our claim in Cases (a) and (b).
	
	In Case (c), the Galois action on $K_1(\ov{F})/2\otimes\CH^1(\ov{X})/2\cong (K_1(\ov{F})/2)^2$ is given by
	\[
	\sigma.(a,\,b)=(\sigma(b),\,\sigma(a))\;,\quad a,\,b\in K_1(\ov{F})/2\;.
	\]
	The map $\delta$ in this case can be identified with the map
	\[
	H^0(F,\,K_1(\ov{F})/2)\lra H^0\bigl(F,\,K_1(\ov{F}/2)\oplus K_1(\ov{F})/2\bigr)
	\]induced from the diagonal map
	\[
	H^0(F,\,K_1(\ov{F})/2)\lra (K_1(\ov{F})/2)^2\;;\quad x\longmapsto (x,\,x)\;.
	\]Therefore, $\delta$ is injective in this case. We conclude that
	$\ker(\mu^{1,\,2})=\ker(\nu^{1,\,2})=0$	in Case (c).
\end{remark}

\begin{thm}\label{6p10temp}
	Assume $\dim\vp>4$. The maps $\mu^{1,\,2}$ and $\eta^3$ are both isomorphisms in each of the following cases:
	
	\begin{enumerate}
		\item $\vp$ is isotropic.
		\item $\vp$ is neither an Albert form nor a Pfister neighbor.
		\item $\dim\vp>8$.
	\end{enumerate}
\end{thm}
\begin{proof}
	We shall use Theorem\;\ref{6p8temp}. In Case (1) the result follows from Prop.\;\ref{3p5temp}. In the other cases we have $\coker(\eta^3)\cong \ker(\cl^2_X)=0$ by Cor.\;\ref{6p2temp} and Prop.\;\ref{6p4temp}. The injectivity of $\eta^3$ is immediate from \cite[Thm.\;5.6]{HLS21}.
\end{proof}

In Theorem\;\ref{6p11temp} below, the result for $\coker(\eta^3)$ is different from its counterpart in characteristic $\neq 2$ (\cite[Thm.\;2 (b)]{Kahn95ArchMath}). The proof given in \cite{Kahn95ArchMath} relies on a description of  $\mathrm{cl}^2_X(h^2-2\ell_1)$ provided in \cite[Prop.\;5.4.6]{Shyevski90}. We do not have a characteristic 2 analog of that result. So we proceed with a different method.

\begin{thm}\label{6p11temp}
	Suppose that $\vp$ is a neighbor of an anisotropic $3$-Pfister form $\Pfi{a,\,b\,;\,c}$.
	
	Then $\ker(\eta^3)$ is generated by $(a)\cup (b)\cup (c]\in H^3(F)$ (so $\ker(\eta^3)\cong \Z/2$), and $\coker(\eta^3)\cong \Z/2$.
\end{thm}
\begin{proof}
	The assertion about $\ker(\eta^3)$ is immediate from \cite[Thm.\;5.6]{HLS21}. To prove the other assertion, we may reduce to the 5-dimensional case thanks to Cor.\;\ref{4p2temp}. We may thus assume $\vp=\dgf{a}\bot \Pfi{b\,;\,c}$. Since there is an injection from $\ker(\eta^3)$ to $\coker(\eta^3)$ by Lemma\;\ref{4p3temp}, it suffices to show $|\coker(\eta^3)|\le 2$.
	
	Let $Y$ be the projective quadric defined by $\Pfi{b\,;\,c}$. It is a smooth hyperplane section in $X$. We have the commutative diagram
	\[
	\begin{CD}
		\CH^2(X)/2 @>i^*>> \CH^2(Y)/2\\
		@V{\cl^2_X}VV @VV{\cl^2_Y}V\\
		H^4(X,\,2) @>>> H^4(Y,\,2)
	\end{CD}
	\]The map $\cl^2_Y$ is injective by Prop.\;\ref{6p4temp}. So the above diagram shows that
	$\ker(\cl^2_X)\subseteq \ker(i^*)$. 	Here
	\[
	\begin{split}
		\CH^2(X)/2&=\CH^2(X)_{\tors}\oplus (\Z/2).h^2=(\Z/2).(h^2-2\ell_1)\oplus(\Z/2).2\ell_1\,,\\
		\CH^2(Y)/2&=(\Z/2).h^2_Y=(\Z/2).2\ell_0\,,
	\end{split}
	\]and the map $i^*$ sends $h^2-2\ell_1$ to $0$ and $2\ell_1$ to $2\ell_0=h^2_Y$. Hence $\ker(i^*)=(\Z/2).(h^2-2\ell_1)=\CH^2(X)_{\tors}$, and by Thm.\;\ref{6p8temp} we have
	\[
	|\coker(\eta^3)|=|\ker(\cl^2_X)|\le |\ker(i^*)|=2
	\]as desired.
\end{proof}

Now we have computed $\ker(\eta^3)$ and $\coker(\eta^3)$ except in the case where $\vp$ is an Albert form. This last case will be treated in \S\,\ref{sec8}.

\section{Unramified Witt groups in characteristic 2}

We need to use residue maps on the Witt group of a discrete valuation field of characteristic 2. We recall some key definitions and facts that will be used in the next section.

Throughout this section, let $K$ be a field extension of $F$ (so $\car(K)=2$) and let $R$ be the valuation ring of a nontrivial discrete valuation $v$ on $K$. Let $\pi\in R$ be a uniformizer and let $k$ be the residue field of $R$.

\medskip

\newpara\label{7p1temp} Let $I_q(R)=W_q(R)$ be the Witt group of nonsingular quadratic spaces over $R$ as defined in \cite[p.18, (I.4.8)]{Baeza78LNM655}. It is naturally a subgroup of $I_q(K)$. For $n\ge 2$, let $I^n_q(R)$ be the subgroup of $I_q(R)$ generated by Pfister forms of the following type:
\[
\langle\!\langle a_1,\cdots, a_{n-1}\,;\,b]]\quad\text{where } a_i\in R^*\,,\,b\in R\,.
\]Put $I^1_q(R)=I_q(R)$. There is a natural homomorphism $I^n_q(R)\to I^n_q(k),\,\vp\mapsto \ov{\vp}$ for each $n\ge 1$.

Following \cite{Arason18}, we define the \emph{tame} (or \emph{tamely ramified}) subgroup of $I_q(K)$ to be the subgroup
$I_q(K)_{tr}:=W(K)\cdot I_q(R)$. For general $n\ge 1$, we put
\[
I^n_q(K)_{tr}:=I^{n-1}(K)\cdot I_q(R)=I^n_q(R)+\langle 1,\,\pi\rangle_{\mathrm{bil}}\cdot I^{n-1}_q(R)\,.
\]
By \cite[Props.\,1.1 and 1.2]{Arason18}, there is a well defined \emph{residue map}
\[
\partial\,:\;I_q(K)_{tr}\lra I_q(k)\,;\quad \vp_0+\pi.\vp_1\longmapsto \ov{\vp}_1\;\text{ for }\,\vp_0,\,\vp_1\in I_q(R)
\]and the sequence
\[
0\lra I_q(R)\lra I_q(K)_{tr}\overset{\partial}{\lra} I_q(k)\lra 0
\]is exact.
For each $n\ge 1$ we have an induced residue map $\partial: I^n_q(K)_{tr}\to I^{n-1}_q(k)$ (with $I^0_q(k)=I^1_q(k)$ by convention), and putting
\[
I^n_q(K)_{nr}:=\ker(\partial: I^n_q(K)_{nr}\to I^{n-1}_q(k))=I^n_q(K)_{tr}\cap I_q(R)\,,
\]we get an exact sequence
\[
0\lra I^n_q(K)_{nr}\lra I_q^n(K)_{tr}\overset{\partial}{\lra} I^{n-1}_q(k)\lra 0\,.
\]

Note that the term ``tame'' and the residue map depend on the discrete valuation $v$ (or the valuation ring $R$).

\

\newpara\label{7p2temp} We use shorthand notation for cohomology functors introduced at the beginning of \S\,\ref{sec5}.

Given $i\in\N$, localization theory in \'etale cohomology theory gives rise to a long exact sequence
\[
\cdots\lra H^{q+1}(R,\,i+1)\lra H^{q+1}(F,\, i+1)\overset{\delta}{\lra} H^{q+2}_{\mathfrak{m}}(R\,,\,i+1)\lra \cdots
\]where $\mathfrak{m}$ denotes the unique closed point of $\Spec(R)$. By \cite[Cor.\;3.4]{Shiho07} we have
\[
  H^{q+2}_{\mathfrak{m}}(R\,,\,i+1)=0\quad\text{ if } q\notin \{i,\,i+1\}\,.
\]Moreover, \cite[Thm.\;3.2]{Shiho07} tells us that
$H^i(k,\,i)\simto H^{i+2}_{\mathfrak{m}}(R\,,\,i+1)$.
So the localization sequence above yields an exact sequence
\[
  \begin{split}
    0&\lra H^{i+1}(R,\,i+1)\lra H^{i+1}(K,\,i+1)\overset{\delta_0}{\lra}H^i(k,\,i)\lra \\
   &\lra H^{i+2}(R)\lra H^{i+2}(K)\overset{\delta_1}{\lra} H^{i+3}_{\mathfrak{m}}(R,\,i+1)\lra H^{i+3}(R,\,i+1) \lra 0\,.
  \end{split}
\]
By the Bloch--Kato--Gabber theorem, the map $\delta_0$ in the above sequence can be identified with the  residue map
$K^M_{i+1}(K)/2\to K^M_i(k)/2$ in Milnor $K$-theory, which is surjective. Hence, we have an exact sequence
\[
0 \lra H^{i+2}(R)\lra H^{i+2}(K)\overset{\delta_1}{\lra} H^{i+3}_{\mathfrak{m}}(R,\,i+1)\lra H^{i+3}(R,\,i+1) \lra 0\,.
\]An element of $H^{i+2}(K)$ is called \emph{unramified} at $v$ if it lies in the subgroup $H^{i+2}(R)=\ker(\delta_1)$.

For each $n\ge 2$, we define the \emph{tame} (or \emph{tamely ramified}) part of $H^n(K)$ to be the subgroup
\[
  H^n_{tr}(K):=\ker\left(H^n(K)\lra H^n(\hat{K})\lra H^n(\hat{K}^{nr})\right)\,,
\]where $\hat{K}$ denotes the $v$-adic completion of $K$ and $\hat{K}^{nr}$ is the maximal unramified extension of $\hat{K}$.
Kato constructed (cf. \cite{Kato82}, \cite[\S\,1]{Ka86}) a \emph{residue map} $\partial^H: H^n_{tr}(K)\to H^{n-1}(k)$ satisfying
\begin{equation}\label{7p2p1temp}
  \partial^H\left((u\pi)\cup (a_1)\cdots (a_{n-2})\cup (b\,]\right)=(\ov{a}_1)\cup \cdots (\ov{a}_{n-2})\cup (\ov{b}\,]\,,\quad u,\,a_i\in R^*\,,\;b\in R\,,
\end{equation}such that the sequence
\[
0\lra H^n(R)\lra H^n(K)_{tr}\overset{\partial^H}{\lra} H^{n-1}(k)\lra 0
\]is exact.
Therefore, an element $\alpha\in H^n(K)$ is unramified if and only if $\alpha\in H^n(K)_{tr}$ and $\partial^H(\alpha)=0$.

\

Using the formula \eqref{7p2p1temp}, one easily proves the following result through calculation.

\begin{prop}\label{7p3temp}
For each $n\ge 2$, we have
\[
e_n(I^n_q(R))\subseteq H^n(R)\,,\;e_n(I^n_q(K)_{tr})\subseteq H^n_{tr}(K)
\]and the following diagram with exact rows is commutative:
\[
  \xymatrix{
 0 \ar[r] & I^n_q(K)_{nr} \ar[d]_{e_n} \ar[rr] && I^n_q(K)_{tr} \ar[d]_{e_n} \ar[rr]^{\partial}  &&  I^{n-1}_q(k) \ar[d]^{e_{n-1}} \ar[r] & 0 \\
 0 \ar[r] & H^{n}(R) \ar[rr] && H^{n}_{tr}(K) \ar[rr]^{\partial^H} && H^{n-1}(k) \ar[r] & 0
  }
\]
\end{prop}

\section{Nontrivial unramified class for Albert quadrics}\label{sec8}

In this section we investigate the case of Albert quadrics and  complete our study of the map $\eta^3$.

\medskip

Recall some standard notation. The hyperbolic plane $\bH$ is the binary quadratic form $(x,\,y)\mapsto xy$. For any nondegenerate form $q$ of dimension $\ge 3$ over $F$, let $F(q)$ denote the function field of the projective quadric defined by $q$.

\begin{lemma}\label{8p1temp}
Let $q$ be an Albert form over $F$ which represents 1, and let $q_1$ be a form over $F(q)$  such that $q_{F(q)}=q_1\bot\bH$.

If $q_1$ represents $1$, then $q$ must be isotropic over $F$.
\end{lemma}
\begin{proof}
  We may write $q=a[1,\,b]\bot c[1,\,d]\bot [1,\,b+d]$, where $a,\,c\in F^*$ and $b,\,d\in F$. Set $q'=q\bot \langle 1\rangle$ and $q''=a[1,\,b]\bot c[1,\,d]\bot \langle 1\rangle$. Since
  \[
  [1,\,b+d]\bot \langle 1\rangle\cong [0,\,b+d]\bot \langle 1\rangle\cong \bH\bot \langle 1\rangle
  \] we have $q'\cong \bH\bot q''$ over $F$. Now
  \[
  \bH\bot q_1\bot \langle 1\rangle=q_{F(q)}\bot \langle 1\rangle=q'_{F(q)}\cong \bH\bot q''_{F(q)}\,.
  \]By Witt cancellation, $q_1\bot \langle 1\rangle\cong q''_{F(q)}$. The assumption that $q_1$ represents 1 implies that $q''_{F(q)}$ is isotropic.

 Assume that  $q$ is anisotropic. Then $q''$ is also anisotropic because it is a subform of $q$. By \cite[Thm.\;1.2 (2)]{Laghribi02Israel}, $q''$ must be a  neighbor of a 3-Pfister form $\pi$. The isotropy of $q''_{F(q)}$ shows that $\pi_{F(q)}$ is hyperbolic. Thus, by \cite[Prop.\;3.4]{Laghribi02Israel}, $q$ must be a neighbor of $\pi$. But this is absurd, because an anisotropic Albert form cannot be a Pfister neighbor (see e.g. \cite[Lemma\;3.4 (1)]{HLS21}).
\end{proof}

Now assume $X$ is defined by  an anisotropic Albert form $\vp$ over $F$.  Let $K=F(X)$ be its function field. We shall normalize $\vp$ so that it represents 1. Thus
\[
\vp=[1,\,b+d]\bot a.[1,\,b]\bot c.[1,\,d]\quad\text{ for some } a,\,b,\,c,\,d\in F^*\,.
\]
We have $\vp=\Pfi{a\,;b}-\Pfi{c\,;d}\in I^2_q(F)$.

Let $\vp_1/K$ be the anisotropic part of $\vp_{K}$. Since $\vp$ is anisotropic of dimension $<8$, the Hauptsatz \cite[(23.7)]{EKM08} implies $\vp\notin I^3_q(F)$. This means that $\vp$ has degree 2 (by \cite[(40.10)]{EKM08}). By \cite[(25.7)]{EKM08}, $\vp_1$ is similar to a 2-Pfister form (and the height of $\vp$ is 2). Let $\tau/K$ be the quadratic 2-Pfister form similar to $\vp_1$, say $\vp_1=f.\tau$, where $f\in K^*$. We have $\tau-\vp_1=\langle 1,\,f\rangle_{\mathrm{bil}}\cdot\tau\in I^3_q(K)$, so that the cohomology class $e_3(\tau-\vp_1)\in H^3(K)$ can be defined.

We claim that $e_3(\tau-\vp_1)$ lies in the unramified cohomology group $H^3_{nr}(X)$. In fact, we can prove the following:

\begin{prop}\label{8p2temp}
  For every discrete valuation $v$ of $K$ that is trivial on $F$, $e_3(\tau-\vp_1)$ is unramified at $v$.
\end{prop}
\begin{proof}
Let $R$ be the valuation ring of $v$ in $K=F(X)$.
 By the commutative diagram in Prop.\;\ref{7p3temp}, it is sufficient to show
 \[
 \tau-\vp_1\in I_q^3(K)_{nr}=I_q(R)\cap I^3_q(K)_{tr}\,.
 \]

 First note that the Witt class of $\vp_1$ is equal to that of $\vp_{K}$. Hence
\[
\vp_1=\Pfi{a\,;b}-\Pfi{c\,;d}\;\in\; \im(I^2_q(F)\lra I_q^2(K))\;\subseteq I^2_q(R)\,.
\]
If $f$ is a unit for $v$, then $\tau-\vp_1=-\langle 1,\,f^{-1}\rangle_{\mathrm{bil}}\cdot\vp_1\in I^3_q(R)\;\subseteq I^3_q(K)_{nr}$.
We may thus assume $\pi:=f^{-1}$ is a uniformizer for $v$. Then
\[
\tau-\vp_1=-\langle 1,\,f^{-1}\rangle_{\mathrm{bil}}\cdot\vp_1=-\langle 1,\,\pi\rangle_{\mathrm{bil}}\cdot\vp_1\;\in\; \langle 1,\,\pi\rangle_{\mathrm{bil}}\cdot I^2_q(R)\;\subseteq I^3_q(K)_{tr}\,.
\]It remains to prove $\tau-\vp_1\in I_q(R)$.  We already know $\vp_1\in I_q^2(R)\subseteq I_q(R)$. So it suffices to show $\tau\in I_q(R)$. This is equivalent to $\tau_{\hat{K}}\in I_q(\hat{R})$, where $\hat{R}$ denotes the completion of $R$ and $\hat{K}$ is the fraction field of $\hat{R}$ (see e.g. \cite[p.106]{Arason18}).

Indeed,
\[
e_2(\vp_1)=e_2(\Pfi{a\,;b}-\Pfi{c\,;d})\;\in\;\im(H^2(F)\lra H^2(K))\;\subseteq H^2(R)\,.
\]
Therefore, the cohomology class $e_2(\vp_1)$ is unramified at $v$. Since $e_2(\tau)=e_2(\vp_1)$ by \eqref{2p1p1temp}, $e_2(\tau)$ is also unramified at $v$. As a 2-Pfister form, $\tau$ is the reduced norm form of a quaternion division $K$-algebra $D$. The cohomology class $e_2(\tau)$ is the Brauer class of $D$. Thus, the fact that $e_2(\tau)$ is unramified means that the quaternion algebra $D$ is unramified. Thus, $D_{\hat{K}}\cong (\al,\,\beta]$ for some $\al\in\hat{R}^*,\,\beta\in \hat{R}$. It follows that $\tau_{\hat{K}}\cong \Pfi{\al\,;\beta}$ (\cite[(12.5)]{EKM08}), proving that the Witt class of $\tau_{\hat{K}}$ lies in $I^2_q(\hat{R})$. In particular, $\tau_{\hat{R}}\in I_q(\hat{R})$ as desired.
\end{proof}

\begin{prop}\label{8p3temp}
  With notation as above,  $e_3(\tau-\vp_1)\notin\im\bigl(\eta^3: H^3(F)\to H^3_{nr}(X)\bigr)$.
\end{prop}
\begin{proof}
Since $e_3: I^3_q(F)\to H^3(F)$ is a surjection and $I^3_q(F)$ is additively generated by 3-Pfister forms, every element of $H^3(F)$ is a sum of finitely many symbols (by a \emph{symbol} in $H^3(F)$ we mean the class $e_3(\pi)$ of a 3-Pfister form $\pi$). We may define the \emph{symbol length} of an element $\beta\in H^3(F)$ to be the smallest positive integer $n$ such that $\beta$ is a sum of $n$ symbols.

We use induction to show:

\emph{For every $n\ge 1$, there is no element $\beta\in H^3(F)$ of symbol length $n$ such that $\beta_{F(X)}=e_3(\tau-\vp_1)$.}

Assume the contrary. Let $\beta\in H^3(F)$ be an element of symbol length $n$ such that $\beta_{F(X)}=e_3(\tau-\vp_1)$.

First assume $n=1$, i.e., $\beta=e_3(\pi)$ for some 3-Pfister form $\pi$ over $F$. Since $\tau\bot-\vp_1$ and $\pi_{F(X)}$ are both 3-Pfister forms, from $e_3(\tau-\vp_1)=e_3(\pi_{F(X)})$ we can conclude that $\tau\bot-\vp_1\cong \pi_{F(X)}$, by \cite[p.237, Prop.\;3]{Kato82}. Letting $E=F(\pi)$,  we get that $\tau-\vp_1=0$ in $I^3_q(E(X))$. Hence $(\vp_1)_{E(X)}=\tau_{E(X)}$ is a 2-Pfister form. In particular, $(\vp_1)_{E(X)}$ represents 1. By Lemma\;\ref{8p1temp}, $\vp_E=\vp_{F(\pi)}$ must be isotropic. But this contradicts an index reduction theorem of Merkurjev (cf. \cite[(30.9)]{EKM08}).

Now consider the general case and write $\beta=\pi+\delta$, where $\pi\in H^3(F)$ is a symbol and $\delta\in H^3(F)$ has symbol length $n-1$. Consider again the function field $E=F(\pi)$.  As before $\vp_E=\vp_{F(\pi)}$ is anisotropic by Merkurjev's index reduction theorem. But
\[
\delta_{E(X)}=\beta_{E(X)}=e_3(\tau_{E(X)}-(\vp_1)_{E(X)})\;\in \;H^3(E(X))\,.
\]This contradicts the induction hypothesis.
\end{proof}

\begin{thm}\label{8p4temp}
  Let $X$ be the projective quadric defined by an anisotropic Albert form over $F$.

  Then the map $\eta^3: H^3(F)\to H^3_{nr}(X)$ is injective and $\coker(\eta^3)\cong \Z/2$.
\end{thm}
\begin{proof}
  The injectivity of $\eta^3$ follows from \cite[Thm.\;5.6]{HLS21}. We have seen in Prop.\;\ref{8p3temp} that $\coker(\eta^3)\neq 0$. It remains to show $|\coker(\eta^3)|\le 2$.

  To this end, we use the isomorphism $\coker(\eta^2)\cong \ker(\cl^2_X)$ obtained in Thm.\;\ref{6p8temp}. By functoriality and the injectivity of $\cl^2_{\ov{X}}$, it follows that
  $\ker(\cl^2_X)$ is contained in the kernel of the restriction map $\xi^2: \CH^2(X)/2\to \CH^2(\ov{X})/2$. By Lemma\;\ref{6p1temp}, we have $\ker(\xi^2)\cong \Z/2$. This completes the proof.
\end{proof}

Together with Prop.\;\ref{4p4temp}, the following corollary also describes the kernel of the cycle class map
\[
\cl^2_{\infty}\,:\; \CH^2(X)\otimes(\Q_2/\Z_2)\lra H^4\bigl(X,\,\Q_2/\Z_2(2)\bigr)\;.
\]

\begin{coro}\label{8p5temp}
Let $X$ be the projective quadric defined by a nondegenerate quadratic form $\vp$ of dimension $\ge 3$.

  Then the map $\eta^3_{\infty}: H^3\bigl(F,\,\Q_2/\Z_2(2)\bigr)\to H^3_{nr}\bigl(X,\,\Q_2/\Z_2(2)\bigr)$ is surjective unless $\vp$ is an anisotropic Albert form. In the latter case $\coker(\eta^3_{\infty})\cong \Z/2$.
\end{coro}
\begin{proof}
If $\dim\vp\le 4$, this has been discussed in Remark\;\ref{5p8temp}. When $\dim\vp>4$, we have already computed $\ker(\eta^3)$ and $\coker(\eta^3)$ in Theorems\;\ref{6p10temp}, \ref{6p11temp} and \ref{8p4temp}. So it suffices to apply Lemma\;\ref{4p3temp}.
\end{proof}

\medskip

\noindent \emph{Acknowledgements.} We are indebted to the referee for carefully reading the manuscript and giving many helpful suggestions. We thank Yang Cao, Ahmed Laghribi, Zhengyao Wu and Yigeng Zhao for helpful discussions. Peng Sun is supported by the National Key R\&D Program of China Grant No.\,2021YFA1001400 and the Fundamental Research Funds for the Central Universities. Yong Hu  is supported by a grant from the National Natural Science Foundation of China (Project No.\,11801260) and the Guangdong Basic and Applied Basic Research Foundation (No.\,2021A1515010396).

\addcontentsline{toc}{section}{\textbf{References}}

\bibliographystyle{alpha}

\bibliography{Unramified}

Contact information of the authors:

\

Yong HU

\medskip

Department of Mathematics

Southern University of Science and Technology


Shenzhen 518055, China


Email: huy@sustech.edu.cn

\

Peng SUN
\medskip

School of Mathematics

Hunan University

Changsha 410082, China


Email: sunpeng@hnu.edu.cn

\medskip

\end{document}